\documentclass[12pt, twoside, leqno]{article}

\usepackage{amsmath,amsthm}
\usepackage{amssymb}

\usepackage{enumerate}

\usepackage{graphicx}

\newcommand{\dotT}{\dot{T}}

\newcommand{\telque}{{\,;\,}}

\newcommand{\bbK}{{\mathbb{K}}}
\newcommand{\bbR}{{\mathbb{R}}}

\newcommand{\bbZ}{{\mathbb{Z}}}

\newcommand{\calD}{{\mathcal{D}}}
\newcommand{\calE}{{\mathcal{E}}}
\newcommand{\calF}{{\mathcal{F}}}

\newcommand{\supp}{{\mathrm{supp\,}}}

\newcommand{\WF}{{\mathrm{WF}}}

\newtheorem{thm}{Theorem}[section]
\newtheorem{cor}[thm]{Corollary}
\newtheorem{lem}[thm]{Lemma}
\newtheorem{prop}[thm]{Proposition}

\theoremstyle{definition}
\newtheorem{dfn}[thm]{Definition}

\numberwithin{equation}{section}

\begin{document}


\baselineskip=17pt


\title{Continuity of the fundamental
operations on distributions having a specified wave front set
(with a counter example by Semyon Alesker)}

\author{Christian Brouder\\
Sorbonne Universit\'es, UPMC Univ. Paris 06,
CNRS UMR 7590,
\\ Mus\'eum National d'Histoire
Naturelle, IRD UMR 206,\\
Institut de Min\'eralogie, de Physique des Mat\'eriaux et de
Cosmochimie, \\
4 place Jussieu, F-75005 Paris, France.\\
\\
Nguyen Viet Dang\\
Laboratoire Paul Painlev\'e (U.M.R. CNRS 8524)\\
Universit\'e de Lille 1\\
59 655 Villeneuve d'Ascq C\'edex France.\\
\\
Fr\'ed\'eric H\'elein\\
Institut de Math\'ematiques de Jussieu Paris Rive Gauche,\\ 
Universit\'e Denis Diderot Paris 7,
B\^atiment Sophie
Germain\\ 75205 Paris Cedex 13, France.}
\date{}

\maketitle


\renewcommand{\thefootnote}{}



\renewcommand{\thefootnote}{\arabic{footnote}}
\setcounter{footnote}{0}


\begin{abstract}
The pull-back, push-forward and multiplication
of smooth functions can be extended to distributions if 
their wave front set satisfies some conditions.
Thus, it is natural to investigate the topological
properties of these operations between 
spaces $\calD'_\Gamma$ of distributions having a wave front set
included in a given closed cone $\Gamma$ of the cotangent space.
As discovered by S. Alesker, the pull-back is not continuous
for the usual topology on $\calD'_\Gamma$,
and the tensor product is not separately continuous.
In this paper, a new topology is
defined for which the pull-back and push-forward
are continuous, the tensor and convolution products and the multiplication
of distributions are hypocontinuous.
\end{abstract}

\section{Introduction}
The motivation
of our work comes
from the renormalization
of QFT in
curved space times,
indeed the
question addressed
in this
paper cannot
be avoided in this context 
and
also the technical
results of this paper
form the core 
of the proof
that perturbative
quantum field theories
are renormalizable
on curved space times \cite{Dangthese,DabrowskiDang}.

 Since L.~Schwartz, we know that the tensor product of distributions
is continuous~\cite[p.~110]{Schwartz-66}
and the product of a distribution by a smooth function
is hypocontinuous~\cite[p.~119]{Schwartz-66} (see definition \ref{hypodef}), although 
it is not jointly continuous~\cite{Kucera-81}.

However, in many applications (for instance the multiplication of 
distributions),
we cannot work with all distributions and we must consider the 
subsets $\calD'_\Gamma$ of distributions 
whose wave front set~\cite{WF1} is included
in some closed subsets $\Gamma$ of 
$\dotT^*\bbR^n=\{(x;\xi)\in T^*\bbR^n\telque \xi\not=0\}$,
where $\Gamma$ is a \emph{cone} in the sense that $(x;\xi)\in\Gamma$
implies $(x;\lambda\xi)\in\Gamma$ for every $\lambda\in\mathbb{R}_{>0}$.
Indeed the spaces $\calD'_\Gamma$ are widely used in 
microlocal analysis because wave front set conditions
rule the so-called \emph{fundamental operations on distributions}:
multiplication, pull-back and push-forward.
The tensor product is also a fundamental operation, but it
holds without condition.

H\"ormander himself, who introduced the concept of a wave front 
set~\cite{Hormander-71}, equipped $\calD'_\Gamma$ with a 
\emph{pseudo-topology}~\cite[p.~125]{Hormander-71},
which is not a topology but
just a rule describing the convergence of sequences.
In particular, when H\"ormander writes that
the fundamental operations are continuous~\cite[p.~263]{HormanderI}, 
he means ``sequentially continuous''.
And indeed, H\"ormander and his followers proved
that, under conditions on the wave front set to
be described later, the following operators
are sequentially continuous: the pull-back of a distribution
by a smooth map~\cite[Thm 8.2.4]{HormanderI};
the push-forward of a distribution by a
proper map~\cite[p.~528]{Chazarain}, the tensor
product of two distributions~\cite[p.~511]{Chazarain}
and the multiplication of two 
distributions~\cite[p.~526]{Chazarain}.

However, sequential continuity was soon found to
be too weak for some applications and 
Duistermaat~\cite[p.~18]{Duistermaat}
equipped $\calD'_\Gamma$ with a locally convex
topology defined in terms of
the following seminorms~\cite[p.~80]{Grigis}:
\begin{itemize}
\item[(i)] All the seminorms on $\calD'(\bbR^n)$ for the weak
topology: $||u||_\phi=|\langle u,\phi\rangle|$ for
all $\phi\in\calD(\bbR^n)$.
\item[(ii)] The seminorms 
$||u||_{N,V,\chi} = \sup_{k\in V} (1+|k|)^N |\widehat{u\chi}(k)|$,
where $N\ge 0$, $\chi\in \calD(\bbR^n)$, and $V\in \bbR^n$ is a closed
cone with $\supp\chi\times V \cap \Gamma=\emptyset$.
\end{itemize}
These seminorms give $\calD'_\Gamma$ the structure of a
locally convex vector space and the corresponding topology is usually
called \emph{H\"ormander's topology}.
It probably first appeared
 in the 1970-1971 lecture notes
by Duistermaat~\cite{Duistermaat}, although the seminorms
$||\cdot||_{N,V,\chi}$ are already mentioned by
H\"ormander~\cite[p.~128]{Hormander-71}.
The (sequential) convergence in the sense
of H\"ormander is~:
a sequence $(u_j)\in \calD'_\Gamma$ 
converges to $u$ in $\calD'_\Gamma$ if and only if
$||u_j-u||_\phi\to 0$ for every $\phi\in \calD(\bbR^n)$
and $||u_j-u||_{N,V,\chi}\to 0$ for 
every $\chi\in \calD(\Omega)$,
every $N\in\mathbb{N}$ and every closed cone
  $V$ in $\bbR^n$ such that 
   $\supp\chi\times V \cap \Gamma=\emptyset$.
Therefore, it is clear that a sequence converges
in the sense of H\"ormander if and only if it
converges in the sense of H\"ormander's topology.

However, for a locally convex space such as
$\calD'_\Gamma$ (which is not metrizable), sequential continuity and 
topological continuity are not equivalent.
Therefore, when Duistermaat states, after defining
the above topology, that 
the pull-back~\cite[p.~19]{Duistermaat},
the push-forward~\cite[p.~20]{Duistermaat} and 
the product of
distributions~\cite[p.~21]{Duistermaat} are continuous,
it is not clear whether he means sequential or topological
continuity. 
When investigating this question for applications
to valuation theory~\cite{Alesker-10}, Alesker discovered
a counterexample proving that the tensor product is
not separately continuous and the pull-back is not
continuous for H\"ormander's topology. In other words,
this topology is too weak to be useful for these questions.

The purpose of the present paper is to 
describe Alesker's counterexample and to
define a topology for which the fundamental operations
have optimal continuity properties:
the tensor product is hypocontinuous,
the pull-back by a smooth map is continuous,
the pull--back
by a family of smooth maps depending 
smoothly on parameters is uniformly
continuous,
the push-forward  by a smooth map is
also continuous, the push--forward by a family
of smooth maps depending
smoothly on parameters is uniformly
continuous, 
the multiplication of distributions
and the
convolution product
are hypocontinuous.
Finally, we discuss how 
the wave front set of distributions
on manifolds can be defined in an intrinsic way.
In appendices, we prove important technical results
concerning the covering of the complement of $\Gamma$,
the topology of $\calD'_\emptyset$ and
the fact that the additional seminorms used to define
the topology of $\calD'_\Gamma$ can be taken to be 
countable.

The main applications of these results are to 
replace technical microlocal proofs by classical topological
statements~\cite{DabrowskiDang,Dangthese}.

\section{Alesker's counterexample}
\label{aleskercounterex}
Semyon Alesker discovered the following counter-example

\begin{prop}
Let $f:\bbR^2\to \bbR$ be the projection to the first 
coordinate. Let $\Gamma=\dotT^*\bbR$, so that
$\calD'_\Gamma(\bbR)=\calD'(\bbR)$, then
$f^*\Gamma=\{(x_1,x_2;\xi_1,0)\}$. We claim that
the map $f^*: \calD'_\Gamma(\bbR) \to \calD'_{f^*\Gamma}(\bbR^2)$
is not topologically continuous for the H\"ormander topology.
\end{prop}
Note that the general definition of
$f^*\Gamma$ is given in Proposition~\ref{pullbackprop}.
\begin{proof}
Let $\varphi\in \calD(\bbR)$ such that 
$\varphi|_{[-1,1]}=1$. Take
$\chi=\varphi\otimes\varphi$, $V=\{(\xi_1,\xi_2)\in \bbR^2
   \telque |\xi_1|\le |\xi_2|\}$ and $N=0$. 
The intersection of $V$ with
$\{(\xi_1,0)\telque \xi_1\not=0\}$ is empty
because $|\xi_1|\le |\xi_2|=0$ implies $\xi_1=\xi_2=0$.
Therefore, 
$||\cdot||_{N,V,\chi}$ is a seminorm of 
$\calD'_{f^*\Gamma}$ and, if $f^*$ were continuous,
it would be possible to bound
$||f^*u||_{N,V,\chi}$ with 
$\sup_i |\langle u,f_i\rangle|$ for a finite set of $f_i\in
\calD(\bbR)$ and every $u\in \calD'(\bbR)$.

We are going to show that this is not the case.
We have
\begin{eqnarray*}
 ||f^* u||_{0,V,\chi} &=& \sup_{\xi\in V} 
  |\widehat{\varphi u}(\xi_1)|\,
|\widehat{\varphi}(\xi_2)|
= 
\sup_{\xi_1} 
  |\widehat{\varphi u}(\xi_1)|\,
  \omega(\xi_1),
\end{eqnarray*}
where  $\omega(\xi_1)=\sup_{|\xi_2|\ge|\xi_1|}
|\widehat{\varphi}(\xi_2)|$.
It is clear that $\omega(\xi_1)>0$ everywhere since
$\widehat{\varphi}$ is a real analytic function.
Thus we should show that the map 
$\calD'(\bbR)\to\bbR$ given by 
$u\mapsto \sup_{\xi\in \bbR} |\widehat{\varphi u}(\xi)|\omega(\xi)$
is not continuous (for a fixed $\omega>0$).

If the pull-back were continuous, there would
be a finite set $\chi_1,\dots,\chi_t$ of functions in $\calD(\bbR)$
such that
\begin{eqnarray*}
||f^* u||_{0,V,\chi} & \le & \sup_{i=1,\dots,t} 
   |\langle u,\chi_i\rangle|.
\end{eqnarray*}
We can find $\xi$ such that the functions 
$\chi_1,\dots,\chi_t$ and $\varphi(x) e^{-ix\xi}$
are linearly independent. Then there exists
$u\in \calD'(\bbR)$ such that
$\langle u,\chi_i\rangle = 0$ for $i=1,\dots,t$
and 
$\widehat{u\varphi}(\xi)=\langle u,\varphi e_\xi\rangle
= 1+ 1/\omega(\xi)$,
where $e_\xi(x)=e^{ -i\xi.x}$.
Then, 
$||f^*u||_{0,V,\chi}= 1 + \omega(\xi)$ and we reach
a contradiction.
\end{proof}
Thus, the pull-back is not continuous.
Moreover, the same example can be considered
as an exterior tensor product
$u\to u\boxtimes 1$. This shows that the exterior tensor
product is not separately continuous
for the H\"ormander topology.

\section{The normal topology and hypocontinuity}
\label{normaltopsect}
We now modify H\"ormander's topology and define
what we call the \emph{normal topology} of 
$\calD'_\Gamma$.  This is a locally convex
topology defined by the same seminorms
$||\cdot||_{N,V,\chi}$ as H\"ormander's
topology, but we replace the seminorms
$||\cdot||_\phi$ of the weak topology of
$\calD'(\bbR^n)$ by the seminorms
$p_B(u)=\sup_{\phi\in B} |\langle u,\phi\rangle|$
(where $B$ runs over the bounded sets of 
$\calD(\Omega)$) of the strong topology
of $\calD'(\bbR^n)$.
The functional properties of this topology, like
completeness, duality, nuclearity,
PLS-property,
bornologicity,
were investigated in detail~\cite{Dabrouder-13}.
As in the case of standard distributions,
several operations will not be jointly
continuous but only hypocontinuous.
Let us recall
\begin{dfn}
\label{hypodef}
\cite[p.~423]{Treves}
Let $E$, $F$ and $G$ be topological vector spaces. 
A bilinear map
$f:E\times F\to G$ is said to be \emph{hypocontinuous}
if: (i) for
every neighborhood $W$ of zero in $G$
and every bounded set $A\subset E$ there is
a neighborhood $V$ of zero in $F$
such that
$f(A\times V)\subset W$ and
(ii) for
every neighborhood $W$ of zero in $G$ and
every bounded set $B\subset F$ there is
a neighborhood $U$ of zero in $E$
such that
$f(U\times B)\subset W$.
\end{dfn}
If $E$, $F$ and $G$ are locally convex spaces
with topologies defined by the families
of seminorms $(p_i)_{i\in I}$, $(q_j)_{j\in J}$ and
$(r_k)_{k\in K}$, respectively, the definition of hypocontinuity 
can be translated into
the following two conditions: (i) For every bounded
set $A$ of $E$ and every seminorm $r_k$, there
is a constant $M$ and a finite set of seminorms
$q_{j_1},\dots,q_{j_n}$ 
(both depending only on $k$ and $A$) 
such that 
\begin{eqnarray}
\forall x\in A, r_k\big(f(x,y)\big) & \le & M 
  \sup\{ q_{j_1}(y),\dots , q_{j_n}(y)\};
\label{hypogenA}
\end{eqnarray}
and (ii) For every bounded
set $B$ of $F$ and every seminorm $r_k$, there
is a constant $M$ and a finite set of seminorms
$p_{i_1},\dots, p_{i_n}$ 
(both depending only on $k$ and $B$) 
such that
\begin{eqnarray}
 \forall y\in B, r_k\big(f(x,y)\big) & \le & M 
  \sup\{ p_{i_1}(x),\dots , p_{i_n}(x)\}.
\label{hypogenB}
\end{eqnarray}
Equivalently~\cite[p.~155]{Kothe-II}, 
we can reformulate 
hypocontinuity
using the 
concept of
equicontinuity~\cite[p.~200]{Horvath}
that is defined as follows~:
\begin{dfn}\label{equicontdefpremiere}
In the general context
of a 
locally 
convex topological vector space
$E$
with seminorms
$(p_\alpha)_{\alpha\in A}$. 
Let $E^*$
be its topological
dual, a set
$H$ in
$E^*$ is called
\emph{equicontinuous}
if and only if the family
of maps
$\ell_v:=u\in E\longmapsto 
\left\langle
u,v
\right\rangle\in \mathbb{R}$
is \textbf{uniformly}
continuous
when $v$ runs
over the set $H$.
\end{dfn}
Hence $f$ is hypocontinuous if
for every bounded set $A$ of $E$ and every
bounded set $B$ of $F$ the sets of maps
$\{f_x\telque x\in A\}$ and 
$\{f_y\telque y\in B\}$ are equicontinuous,
where $f_x: E\to G$ and $f_y: F\to G$ are defined
by $f_x(y)=f_y(x)=f(x,y)$.



\section{Tensor product of distributions}
\label{tensprodsect}

Let $\Omega_1$ and $\Omega_2$ be open sets in 
$\bbR^{d_1}$ and $\bbR^{d_2}$, respectively, and
$(u,v)\in \calD'_{\Gamma_1}\times
\calD'_{\Gamma_2}$,
where $\calD'_{\Gamma_1}\subset\calD'(\Omega_1)$ and
$\calD'_{\Gamma_2}\subset\calD'(\Omega_2)$. Then the
tensor product $u\otimes v$ belongs to 
$\calD'_\Gamma\subset \calD'(\Omega_1\times\Omega_2)$ 
where 
\begin{eqnarray*}
\Gamma &= & 
\big(\Gamma_1 \times \Gamma_2\big) \cup 
\big((\Omega_1\times\{0\})\times \Gamma_2\big)
\cup \big(\Gamma_1\times (\Omega_2\times\{0\})\big)
\\&=&
\left(\Gamma_1\cup\{\underline{0}\}_1\right)\times 
  \left(\Gamma_2\cup\{\underline{0}\}_2\right)\setminus 
  \{(\underline{0},\underline{0})\},
\end{eqnarray*}
$\{\underline{0}\}_1$ means $\Omega_1\times\{0\}$,
$\{\underline{0}\}_2$ means $\Omega_2\times\{0\}$
and $\{\underline{0},\underline{0}\}$ 
means $(\Omega_1\times\Omega_2)\times\{0,0\}$.
Our goal in this section is to show that the 
tensor product is hypocontinuous for the
normal topology.
We denote by $(z;\zeta)$ the coordinates
in $T^*(\Omega_1\times\Omega_2)$,
where $z=(x,y)$ with $x\in \Omega_1$ and $y\in \Omega_2$,
$\zeta=(\xi,\eta)$ with 
$\xi\in \bbR^{d_1}$ and $\eta\in \bbR^{d_2}$.
We also denote $d=d_1+d_2$, so that $\zeta\in \bbR^d$.

\begin{lem}
The seminorms of the strong topology of $\calD'(\bbR^d)$
and the family of seminorms:
\begin{equation}
\Vert t_1\otimes t_2\Vert_{N,V,\varphi_1\otimes
\varphi_2}=\sup_{\zeta\in V}
(1+|\zeta|)^N\vert \widehat{t_1\varphi_1}(\xi)\vert\,\vert
   \widehat{t_2\varphi_2}(\eta)\vert,
\end{equation}
where $\zeta=(\xi,\eta)$,
$(\varphi_1,\varphi_2)\in\calD(\Omega_1)\times \calD(\Omega_2)$ 
and $V\subset\mathbb{R}^{d}$, are such that
$\left(\supp\left(\varphi_1\otimes\varphi_2\right)\times V\right)
\cap \Gamma=\emptyset$,
are a fundamental system of seminorms for the normal
topology of $\calD'_\Gamma$.
\end{lem}
\begin{proof}

We use the following lemma~\cite[p.~80]{Grigis}
\begin{lem}
Let $\Omega$ be an open set in $\bbR^n$.
If we have a family, indexed by $\alpha\in A$,
of $\chi_\alpha\in \calD(\Omega)$
and of closed cones $V_\alpha\subset (\bbR^n\backslash\{0\})$
such that 
$(\supp\chi_\alpha\times V_\alpha)\cap \Gamma=\emptyset$ and
\begin{eqnarray*}
\Gamma^c &=& \bigcup_{\alpha\in A} \{(x,\xi)\in \dotT^*\Omega
  \telque \chi_\alpha(x)\not=0, \xi\in \mathring{V_\alpha}\},
\end{eqnarray*}
then the topology of $\calD'_\Gamma$ is already defined by
the strong topology of $\calD'(\Omega)$ and 
the seminorms $||\cdot||_{N,V_\alpha,\chi_\alpha}$.
\end{lem}
It is clear that the family indexed by $\varphi_1\otimes\varphi_2$
and $V$ such that $\supp(\varphi_1\otimes\varphi_2)\times
V\cap\Gamma=\emptyset$ satisfies the hypothesis of the
lemma.
\end{proof}

To establish the hypocontinuity 
of the tensor product,
we consider an arbitrary bounded set 
$B\subset\calD'_{\Gamma_1}(\Omega_1)$ and,
according to eq.~\eqref{hypogenA}, we must
show that, for every seminorm $r_k$ of 
$ \calD'_\Gamma(\Omega_1\times\Omega_2)$, there is a constant $M$ and a 
finite number of seminorms $q_j$ such that
$r_k(u\otimes v)\leqslant M\sup_j q_j(v)$ for 
every $u\in  B$ and every $v\in \calD'_{\Gamma_2}(\Omega_2)$.
By Schwartz' 
theorem ~\cite[p.~110]{Schwartz-66} we already know that this
is true for every seminorm $r_k$ of the 
strong topology of $\calD'_\Gamma(\Omega_1\times\Omega_2)$. It remains
to show it for every
$||\cdot||_{N,V,\varphi_1\otimes\varphi_2}$.
This will be done by first defining 
a suitable partition of unity on $\Omega_1\times \Omega_2$ 
and its corresponding cones. Then, this partition of unity
will be used to bound the seminorms by standard microlocal
techniques.

\begin{lem}\label{lemm0}
Let $\Gamma_1,\Gamma_2$ be closed cones in 
$\dotT^*\Omega_1$ and $\dotT^*\Omega_2$, respectively.
Set $\Gamma=\left(\Gamma_1\cup\{\underline{0}\}\right)\times 
\left(\Gamma_2\cup\{\underline{0}\}\right)\setminus 
\{(\underline{0},\underline{0})\}\subset \dotT^*\bbR^d$.
Then for all closed cones $V\subset\bbR^{d}$ and 
$\chi\in\calD(\Omega_1\times\Omega_2)$ 
such that
$\left(\supp\chi\times V\right)\cap\Gamma=\emptyset$,
there exist a partition of unity 
$(\psi_{j1}\otimes\psi_{j2})_{j\in J}$ 
of $\Omega_1\times\Omega_2$,
which is finite on $\supp\chi$,
and a family of closed cones $(W_{j1}\times W_{j2})_{j\in J}$ in 
$(\bbR^{d_1}\setminus\{0\})
\times (\bbR^{d_2}\setminus\{0\})$ such that
\begin{eqnarray}
(\supp\psi_{j1}\times W_{j1}^c)\cap \Gamma_1 &=& 
  (\supp\psi_{j2}\times W_{j2}^c)\cap \Gamma_2
= \emptyset,\\
V\cap\left((W_{j1}\cup\{0\})\times(W_{j2}\cup\{0\})\right) &=&\emptyset,
\\
\quad\text{ if }\,
\supp\chi\cap\supp(\psi_{j1}\otimes\psi_{j2})\not=\emptyset.
\nonumber
\end{eqnarray}
\end{lem}
\begin{proof} --- We first set some notation.
For any $D\in \mathbb{N}$, with the identification
$T^*\mathbb{R}^D\simeq \mathbb{R}^D\oplus (\mathbb{R}^D)^*$, we denote by
$\underline{\pi}:T^*\mathbb{R}^D\longrightarrow \mathbb{R}^D$
the projection onto the first factor and by
$\overline{\pi}:T^*\mathbb{R}^D\longrightarrow (\mathbb{R}^D)^*$
the projection on the second factor.
We use the distance
$d_\infty$ on $\mathbb{R}^D$ (or $(\mathbb{R}^D)^*$) defined by 
$d_\infty(u,v):= \sup_{1\leq i\leq D}|u^i-v^i|$. For $u\in \mathbb{R}^D$
and $r\geq 0$ we then set $\overline{B}(u,r) = \{v\in \mathbb{R}^D;
d_\infty(u,v)\leq r\}$ and, for any subset $Q\subset \mathbb{R}^D$, 
$Q_{,r}:= \{v\in \mathbb{R}^D;d_\infty(v,Q)\leq r\}$.
We note that, for any
pair of sets $Q_1\subset \mathbb{R}^{d_1}$
and $Q_2\subset \mathbb{R}^{d_2}$, 
$(Q_1\times Q_2)_{,r} = Q_{1,r}\times Q_{2,r}$ (in particular,
if $(x,y)\in \Omega_1\times \Omega_2$,
$\overline{B}((x,y),r) = \overline{B}(x,r) \times \overline{B}(y,r)$).
Lastly for any closed conic subset
$W\subset (\mathbb{R}^D)^*\setminus \{0\}$, 
we set $\overline{W}:= W\cup \{0\}$
for short and $UW:= S^{D-1}\cap W$. Similarly if $\Gamma$ is a conic
subset of $T^*\mathbb{R}^D$, we set $ U\Gamma = (\mathbb{R}^D\times S^{D-1})\cap \Gamma$
and $\overline{\Gamma}=\Gamma\cup \underline{0}\subset T^*\mathbb{R}^D$
where $\underline{0}$ is the zero
section of $T^*\mathbb{R}^D$.

We will prove that there exists a family of open balls
$(B_{j1}\times B_{j2})_{j\in J}$ that covers
$\Omega_1\cap\Omega_2$, which is finite over any compact subset
of $\Omega_1\times \Omega_2$ and in particular on $\supp\chi$
and such that
$(\overline{B_{j1}}\times W_{j1}^c)\cap \Gamma_1 =
(\overline{B_{j2}}\times W_{j2}^c)\cap \Gamma_2 = \emptyset$ and that
$V\cap(\overline{W}_{j1}\times \overline{W}_{j2}) = \emptyset$,
if $\supp\chi\cap(\overline{B}_{j1}\times \overline{B}_{j2}) \neq \emptyset$.
The conclusion of the lemma will then follow 
by constructing a partition of unity 
$(\psi_{j1}\otimes\psi_{j2})_{j\in J}$ 
such that $\supp\psi_{j1} = \overline{B}_{j1}$ and
$\supp\psi_{j2} = \overline{B}_{j2}$, $\forall j\in J$,
by using standard arguments.

\emph{Step 1.} If $(\supp\chi\times V)\cap \Gamma = \emptyset$, then
there exists some $\delta>0$ such that
$d_\infty(\supp\chi\times UV,U\Gamma) \geq 4\delta$. Consider
$K:= (\supp\chi)_{,\delta}$, we then note that
$d_\infty(K\times UV,U\Gamma) \geq 3\delta$. Without loss of generality, we can assume
that $\delta$ has been chosen so that $K\subset \Omega_1\times \Omega_2$.
Obviously $\Omega_1\times \Omega_2$ is covered by
$(B((x,y),\delta))_{(x,y)\in \Omega_1\times \Omega_2}$.
Moreover all balls $B((x,y),\delta)$ are contained in $K$ if
$(x,y)\in \supp\chi$ and $\supp\chi$ is covered by the subfamily
$(B((x,y),\delta))_{(x,y)\in \supp\chi}$. Since $\supp\chi$ is
compact we can thus extract a countable family of balls
$(B_i)_{i\in I} = (B_{i1}\times B_{i2})_{i\in I}$ which covers
$\Omega_1\times \Omega_2$ and which is
finite over $\supp\chi$.

We now set
$\gamma:= \overline{\pi}(\underline{\pi}^{-1}(K)\cap \Gamma)$
and $U\gamma:= \overline{\pi}(\underline{\pi}^{-1}(K)\cap U\Gamma)$
and we estimate the distance of $U\gamma$ to $UV$:
\[
 \begin{array}{ccl}
  d_\infty[U\gamma,UV]
  & = & \displaystyle \inf_{\xi\in \overline{\pi}(\underline{\pi}^{-1}(K)\cap U\Gamma)}
  \quad \inf_{\eta\in UV}d_\infty(\xi,\eta)\\
  &  = & \displaystyle \inf_{(u,\xi)\in U\Gamma;u\in K}
  \quad \inf_{(v,\eta)\in K\times UV}d_\infty(\xi,\eta)\\
  & = & \displaystyle \inf_{(u,\xi)\in U\Gamma;u\in K}
  \quad \inf_{(v,\eta)\in K\times UV}d_\infty((u,\xi),(v,\eta)),
 \end{array}
\]
where the last equality is due to the fact that one can choose $v=u$ in the minimization.
We deduce that, by removing the constraint $u\in K$ in the minimization,
\[
 \begin{array}{ccl}
  d_\infty[U\gamma,UV]
  & \geq & \displaystyle \inf_{(u,\xi)\in U\Gamma}
  \quad \inf_{(v,\eta)\in K\times UV}d_\infty((u,\xi),(v,\eta))\\
  & = & d_\infty(K\times UV,U\Gamma) \geq 3\delta.
 \end{array}
\]
\emph{Step 2.} Since $\gamma$
and $V$ are cones, the previous inequality implies
$d_\infty(\xi,V)\geq 2\|\xi\|\delta$ for every $\xi\in \gamma$.
For any $i\in I$ such that the ball $B_i$ is centered at a point
in $\supp\chi$, the inclusion $\overline{B_i}\subset K$ implies
$\overline{\pi}(\underline{\pi}^{-1}(\overline{B_i})\cap \Gamma)\subset \gamma$.
We hence have also
\begin{equation}\label{distance-securite}
\forall \xi\in \overline{\pi}(\underline{\pi}^{-1}(\overline{B_i})\cap \Gamma)
\quad d_\infty(\xi,V)\geq 2\|\xi\|\delta.
\end{equation}
We now set $\overline{W}_{i1}:=
\{\xi_1\in (\mathbb{R}^{d_1})^*;
d_\infty(\xi_1,\overline{\pi}(\underline{\pi}^{-1}(\overline{B_{i1}})\cap \Gamma_1))
\leq \|\xi_1\|\delta\}$,
$\overline{W}_{i2}:=
\{\xi_2\in (\mathbb{R}^{d_2})^*;
d_\infty(\xi_2,\overline{\pi}(\underline{\pi}^{-1}(\overline{B_{i2}})\cap \Gamma_2))
\leq \|\xi_2\|\delta\}$
and $W_{i1}:= \overline{W}_{i1}\setminus\{0\}$, $W_{i2}:= \overline{W}_{i2}\setminus\{0\}$.
By the definition of $W_{i1}$, $W_{i1}^c\cap \overline{\pi}(\underline{\pi}^{-1}(\overline{B_{i1}})\cap \Gamma_1) = \emptyset$,
which is equivalent to $(\overline{B}_{i1}\times W_{i1}^c)\cap \Gamma_1 = \emptyset$. Similarly
$(\overline{B}_{i2}\times W_{i1}^c)\cap \Gamma_2 = \emptyset$.

On the other hand, since 
\begin{eqnarray*}
\overline{\pi}(\underline{\pi}^{-1}(\overline{B_{i1}})\cap\overline{\Gamma_1})
\times \overline{\pi}(\underline{\pi}^{-1}(\overline{B_{i2}})\cap\overline{\Gamma_2})
&=& \overline{\pi}[\underline{\pi}^{-1}(\overline{B_{i1}}\times \overline{B_{i2}})\cap(\overline{\Gamma_1}\times \overline{\Gamma_2})]\\
&=& \overline{\pi}[\underline{\pi}^{-1}(\overline{B_i})\cap\overline{\Gamma}],\,\ 
\overline{B_i}=\overline{B_{i1}}\times \overline{B_{i2}}
\end{eqnarray*}
because
\[
\{\xi_1;\exists (x_1;\xi_1)\in\overline{\Gamma_1} , x_1\in \overline{B_{i1}} \}
\times
\{\xi_2;\exists (x_2;\xi_2)\in\overline{\Gamma_2} , x_2\in \overline{B_{i2}} \}\]
\[=\{(\xi_1,\xi_2);
\exists (x_1,x_2;\xi_1,\xi_2)\in\overline{\Gamma_1}\times\overline{\Gamma_2},
(x_1,x_2)\in \overline{B_{i1}}\times \overline{B_{i2}}\}
\]
we also have
\begin{eqnarray*}
 \overline{W}_{i1}\times \overline{W}_{i2}& =& \{(\xi_1,\xi_2)\in (\mathbb{R}^d)^*;
 d_\infty[(\xi_1,\xi_2),\overline{\pi}(\underline{\pi}^{-1}(\overline{B_i})\cap\Gamma)]\\
 &\leq &\sup (\|\xi_1\|,\|\xi_2\|) \delta\}.
\end{eqnarray*}
Hence by (\ref{distance-securite}), we deduce that $\overline{W}_{i1}\times \overline{W}_{i2}$ does not meet $V$.
\end{proof}
In the
rest  of the 
paper, we may identify 
abusively
$\bbR^d$ and $(\bbR^d)^*$.
 We also introduce
the notation
$e_\zeta(x,y)=e^{i(\xi.x+\eta.y)}$
where $\zeta=(\xi,\eta)$.
To estimate $||u\otimes v||_{N,V,\varphi_1\otimes\varphi_2}$,
we use Lemma~\ref{lemm0} to find a partition of unity
$(\psi_{j1}\otimes\psi_{j2})_{j\in J}$ which is finite
on $\supp(\varphi_1\otimes\varphi_2)$ to write
\begin{eqnarray*}
\widehat{u\varphi_1}(\xi) 
\widehat{v\varphi_2}(\eta) =
\calF(u\varphi_1\otimes v\varphi_2)(\zeta) &=&
\langle u\otimes v,(\varphi_1\otimes \varphi_2) e_\zeta\rangle
\\=\sum_j 
\langle u\otimes v,(\varphi_1\psi_{j1}\otimes \varphi_2\psi_{j2}) 
e_\zeta\rangle
&=& \sum_j \widehat{u\varphi_1\psi_{j1}}(\xi) \,
 \widehat{v\varphi_2\psi_{j2}}(\eta).
\end{eqnarray*}
Therefore
$||u\otimes v||_{N,V,\varphi_1\otimes\varphi_2}
\leqslant \sum_j ||u\otimes v||_{N,V,\varphi_1\psi_{j1}\otimes
  \varphi_2\psi_{j2}}$, where the sum over $j$ is finite.
Each seminorm on the right hand side is bounded by the
following lemma.
\begin{lem}
Let $\Gamma_1$, $\Gamma_2$ and $\Gamma$ be closed
cones as in the previous lemma, $\psi_1\in \calD(\Omega_1)$
$\psi_2\in \calD(\Omega_2)$ such that
$ (\supp(\psi_1\otimes\psi_2) \times V ) \cap\Gamma=\emptyset$
and closed cones
$W_1$ and $W_2$ in $\bbR^{d_1}\backslash\{0\}$ and
$\bbR^{d_2}\backslash\{0\}$ such that
\begin{eqnarray}
(W_1\cup\{0\})\times (W_2\cup\{0\})\cap V &=& \emptyset,\\
(\supp\psi_k\times W_k^c)\cap\Gamma_k &=& \emptyset,
\text{ for }\, k=1,2.
\end{eqnarray}
Then, for every bounded set $A\subset\calD'_{\Gamma_1}$
and every integer $N$, there are constants
$m$, $M_1$, $M_2$ and a bounded set
$B\subset \calD(K)$, where $K$ is an arbitrary
compact neighborhood of $\supp\psi_2$,  such that
\begin{eqnarray*}
||t_1\otimes t_2||_{N,V,\psi_1\otimes\psi_2} & \le &
  M_1 
   ||t_2||_{N,C_\beta,\psi_2}
+ 
 M_2
   ||t_2||_{N+m,C_\beta,\psi_2} 
+ p_{B}(t_2),
\end{eqnarray*}
for every $t_1\in A$ and $t_2\in \calD'_{\Gamma_2}$,
where $C_\beta$ is an arbitrary conic neighborhood
of $W_2$ with compact base 
and $p_B$ is a seminorm of the strong topology of
$\mathcal{D}^\prime(\bbR^{d_2})$.
\end{lem}
\begin{proof}
We want to calculate
\begin{eqnarray*}
||t_1\otimes t_2||_{N,V,\psi_1\otimes\psi_2}
= \sup_{\zeta\in V}  
    (1+|\zeta|)^N
    |\calF(t_1\psi_1\otimes t_2\psi_2)(\zeta)|.
\end{eqnarray*}    
We denote $u=t_1\psi_1$, $v=t_2\psi_2$
and $I=\widehat{u\otimes v}$. From $e_{(\xi,\eta)}=e_\xi\otimes e_\eta$
we find that
$I(\xi,\eta)=\langle t,e_{(\xi,\eta)}\rangle
=\langle u\otimes v, e_\xi\otimes e_\eta\rangle
=\langle u, e_\xi\rangle\langle v, e_\eta\rangle
=\widehat{u}(\xi)\widehat{v}(\eta)$.
By the shrinking lemma we can slightly enlarge
$W_1$ and $W_2$ to closed cones having the same properties.
Thus, there are two homogeneous functions of degree zero
$\alpha$ and $\beta$ on $\bbR^{d_1}$ and $\bbR^{d_2}$,
respectively, which are smooth except at the origin, 
non-negative and bounded by 1, such that:
(i) $\alpha|_{W_1\cup\{0\}}=1$ and $\beta|_{W_2\cup\{0\}}=1$;
(ii) $(\supp\alpha\times\supp\beta)\cap V=\emptyset$;
(iii) $(\supp\psi_1\times \supp(1-\alpha))\cap \Gamma_1=\emptyset$;
(iv) $(\supp\psi_2\times \supp(1-\beta))\cap \Gamma_2=\emptyset$.
We can write $I=I_1+I_2+I_3+I_4$ where (recalling that 
$\zeta=(\xi,\eta)$)
\begin{eqnarray*}
I_1(\zeta) &=&
\alpha(\xi)\widehat{u}(\xi)\beta(\eta)\widehat{v}(\eta),\\
I_2(\zeta) &=&
\alpha(\xi)\widehat{u}(\xi)(1-\beta)(\eta)\widehat{v}(\eta),\\
I_3(\zeta) &=&
(1-\alpha)(\xi)\widehat{u}(\xi)\beta(\eta)\widehat{v}(\eta),\\
I_4(\zeta) &=&
(1-\alpha)(\xi)\widehat{u}(\xi)(1-\beta)(\eta)\widehat{v}(\eta).
\end{eqnarray*}
The term $I_1(\zeta)=0$ because, by condition (ii)
$\alpha(\xi)\beta(\eta)=0$ for $(\xi,\eta)\in V$.
Condition (iii)  implies that
\begin{eqnarray*}
|(1-\alpha)(\xi)\widehat{u}(\xi)| 
  &\le & \sup_{\xi\in C_\alpha} 
   |\widehat{t_1\psi_1}(\xi)| 
  \le (1+|\xi|)^{-N} ||t_1||_{N,C_\alpha,\psi_1},
\end{eqnarray*}
where $\xi\in C_\alpha=\supp(1-\alpha)$.
This gives us, with $C_\beta=\supp(1-\beta)$,
\begin{eqnarray*}
|I_4(\zeta)| & \le & 
(1+|\xi|)^{-N}(1+|\eta|)^{-N} 
   ||t_1||_{N,C_\alpha,\psi_1}
   ||t_2||_{N,C_\beta,\psi_2}
\\&\le &
(1+|\zeta|)^{-N}
   ||t_1||_{N,C_\alpha,\psi_1}
   ||t_2||_{N,C_\beta,\psi_2},
\end{eqnarray*}
because 
$1+|(\xi,\eta)|\le 1 + |\xi|+ |\eta|\le 
(1+|\xi|)(1+|\eta|)$.
Since the set $A$ is bounded in $\calD'_{\Gamma_1}$
there is a constant $M_1=\underset{t_1\in A}{\sup}\Vert t_1\Vert_{N,C_\alpha,\psi_1}$ such that
$|I_4(\zeta)| \le (1+|\zeta|)^{-N} M_1 ||t_2||_{N,C_\beta,\psi_2}$.

To estimate $I_2$, we use the fact that, $u=t_1\psi_1$ being a compactly
supported distribution there is an integer $m$ such that,
for all $t_1\in A$,
\begin{eqnarray*}
|\alpha(\xi) \widehat{u}(\xi)| \le
|\widehat{u}(\xi)| \le
(1+|\xi|)^m ||\theta^{-m} \widehat{u}||_{L^\infty}.
\end{eqnarray*}
As for the estimate of $I_4$, we get
$|(1-\beta)(\eta)\widehat{v}(\eta)| 
  \le (1+|\eta|)^{-N-m} ||t_2||_{N+m,C_\beta,\psi_2}.$
The set 
$\overline{\{\zeta\in \supp\alpha\times
C_\beta\telque |\zeta|=1\}\cap V}$
is \textbf{compact} and
avoids the set of all
elements of the
form $\zeta=(\xi,0),
\xi\in\text{supp }\alpha\setminus\{0\}$. 
Otherwise, we would find some
sequence $(\xi_n,\eta_n)\rightarrow (\xi,0)
\in  ((\text{supp }\alpha\times \{0\})\cap V)\subset
((\text{supp }\alpha\times \text{supp }\beta)\cap V)$
which contradicts
the condition (ii).
Let $\epsilon>0$ be 
the smallest value 
of $|\eta|$ in this set.
Then, the functions 
$\alpha$ and $\beta$ being
homogeneous of degree zero, 
$\supp\alpha\times C_\beta\cap V$ is a cone
in $\bbR^d$ and 
$|\eta|/|\zeta|  \ge \epsilon$ for all
$\zeta=(\xi,\eta)$ in the set
$\supp\alpha\times C_\beta\cap V$.
Thus, $(1+|\eta|)^{-N-m} \le \epsilon^{-N-m}
 (1+|(\xi,\eta)|)^{-N-m}$ and 
$|I_2(\zeta)| \le  ||\theta^{-m} \widehat{t_1\psi_1}||_{L^\infty}
   ||t_2||_{N+m,C_\beta,\psi_2} \epsilon^{-N-m}
(1+|\zeta|)^{-N}$,
for every $\zeta\in V$,
because $|\xi|\le|(\eta,\xi)|$.
We now prove an intermediate lemma:
\begin{lem}
\label{Linflem}
Let $\Omega$ be an open set of $\bbR^d$
and $B$ a bounded set in $\mathcal{D}'(\Omega)$, then for every
$\chi\in \mathcal{D}(\Omega)$ there exist an integer $M$ 
and a constant $C$ (both depending only on $B$ and on
an arbitrary relatively compact open neighborhood of $\supp\chi$)
such that
\begin{eqnarray*}
\sup_{u\in B}\sup_{\xi\in
\bbR^n}(1+|\xi|)^{-M}|\widehat{u\chi}(\xi)| < 
 2^M C\, \mathrm{Vol}(K)\, \pi_{M,K}(\chi),
\end{eqnarray*}
where $\pi_{m,K}(\chi)=\underset{x\in K,\vert\alpha\vert\leqslant M}{\sup}\vert\partial^\alpha\chi(x)\vert$ and for $K=\supp\chi$.
\end{lem}
\begin{proof}
Let $\Omega_0$ be a relatively compact open
neighborhood of $K=\supp \chi$.
According to Schwartz~\cite[p.~86]{Schwartz-66},
for any bounded set $B$ in $\calD'(\Omega)$,
there is an integer $M$ (depending only on $B$
and $\Omega_0$) such that every $u\in B$
can be expressed in $\Omega_0$ as $u=\partial^\alpha f_u$
for $|\alpha|\le M$, where $f_u$ is a continuous function.
Moreover, there is a constant $C$ (depending only on
$B$ and $\Omega_0$) such that $|f_u(x)|\le C$ for all $x\in \Omega_0$
and $u\in B$.
Thus,
\begin{eqnarray*}
\widehat{u\chi}(\xi) &=&
\int_{\Omega_0} e^{ -i\xi\cdot x} \chi(x) \partial^\alpha f_u(x) dx
= (-1)^{|\alpha|}
\int_{\Omega_0} f_u(x) \partial^{\alpha}
    \big(e^{ -i\xi\cdot x} \chi(x)\big) dx
\\&=&
 (-1)^{|\alpha|}
\sum_{\beta\le\alpha} \binom{\alpha}{\beta}
  (i\xi)^\beta
  \int_{\Omega_0} f_u(x) e^{ -i\xi\cdot x}
   \partial^{\beta-\alpha}\chi(x) dx.
\end{eqnarray*}
By using $|(i\xi)^\beta|\le (1+|k|)^M$ if $|\beta|\le M$ we obtain
\begin{eqnarray*}
(1+|\xi|)^{-M}|\widehat{u\chi}(\xi) |&\le &
\sup_{|\alpha|\le M}
\sum_{\beta\le\alpha} \binom{\alpha}{\beta}
  \Big|\int_{\Omega_0} f_u(x) e^{ -i\xi\cdot x}
\partial^{\beta-\alpha}\chi(x)
dx\Big|
\\ &\le &
 2^M C\, \mathrm{Vol}(K)\, \pi_{M,K}(\chi).
\end{eqnarray*}
\end{proof}
By lemma~\ref{Linflem}, 
$||\theta^{-m} \widehat{t_1\psi_1}||_{L^\infty}$
is uniformly bounded for 
$t_1\in A$ by a constant 
$M'_2=\underset{t_1\in A}{\sup}\Vert\theta^{-m}\widehat{t_1\psi_1}\Vert_{L^\infty},
\theta=1+\vert\xi\vert$.
Therefore, there is a constant 
$M_2=M'_2 \epsilon^{-N-m}$ such that,
for every $t_1\in A$ and 
every $t_2\in \calD'_{\Gamma_2}$,
$|I_2(\zeta)| \le M_2 
||t_2||_{N+m,C_\beta,\psi_2}$.

The term $I_3$ is treated differently because we want to
get the following result:
for every bounded set $A$ in $\calD'_{\Gamma_1}$ 
and every seminorm $||\cdot||_{N,V,\chi}$,
there is a bounded set $B\in \calD(\Omega_2)$
such that for all
$\zeta\in V, I_3(\zeta) 
  \le p_B(t_2) (1+\vert\zeta\vert)^{-N}$ for every 
$t_2\in \calD'_{\Gamma_2}$.
This special form of eq.~\eqref{hypogenA}
is possible because the
union of bounded sets is a bounded set and
the multiplication of a bounded set by
a positive constant $M$ is a bounded set.

We write
$I_3(\zeta)=\langle t_2, f_{\zeta}\rangle$,
where 
$f_{(\xi,\eta)}(y)=(1-\alpha)(\xi)\widehat{u}(\xi)
\beta(\eta) \psi_2(y) e_\eta(y)$ and
we must show that the set
$B=\{(1+|\zeta|)^N f_{\zeta}\telque \zeta\in V\}$ is
a bounded set of $\calD(\Omega_2)$.
A subset $B$ of $\calD(\Omega_2)$ is bounded if and only if
there is a compact set $K$ and
a constant $M_n$ for every integer $n$ such that
$\supp f\subset K$ and
$\pi_{m,K}(f)\le M_n$ for every $f\in B$.
All $f_\zeta$ are supported on 
$\supp\psi_2$ and are smooth functions
because $\psi_2$ and $e_\eta$ are smooth.
We have to prove that, if $t_1$ runs over a bounded
set of $\calD'_{\Gamma_1}$, then there are
constants $M_n$ such that
$\pi_{n,K}(f_\zeta)\le M_n$ for all
$\zeta\in V$, where $K$ is a compact
neighborhood of $\supp\psi_2$.
We start from
\begin{eqnarray*}
\pi_{n,K}(f_{(\xi,\eta)})  &=& 
\big|(1-\alpha)(\xi)\widehat{u}(\xi)
\beta(\eta)\big| \pi_{n,K}(\psi_2 e_\eta).
\end{eqnarray*}
We notice that
$\pi_{n,K}(\psi_2 e_\eta) \le 2^n\pi_{n,K}(\psi_2)
\pi_{n,K}(e_\eta)$ and that
$\pi_{n,K}(e_\eta) \le |\eta|^n$.
As for the estimate of $I_2$, we have
$|(1-\alpha)(\xi)\widehat{u}(\xi)|\le (1+|\xi|)^{-N-n}
||t_1||_{N+n,C_\alpha,\psi_1}$ because
$(\supp\varphi_1\times \supp(1-\alpha)) \cap \Gamma_1=\emptyset$
and $(1+|\xi|)^{-N-n}\le \epsilon^{-N-n}
(1+|(\xi,\eta)|)^{-N-n}$
for
some $\varepsilon$
because $(\xi,\eta)\in (V\cap \supp(1-\alpha)\times\supp\beta)$.
Therefore
\begin{eqnarray*}
\pi_{n,K}(f_\zeta) & \le &
  ||t_1||_{N+n,C_\alpha,\psi_1} 2^n \pi_{n,K}(\psi_2)
 \epsilon^{-N-n} (1+|\zeta|)^{-N},
\end{eqnarray*}
because $|\eta|^n(1+|(\xi,\eta)|)^{-n}\le 1$.
If $t_1$ belongs to a bounded set $A$ of
$\calD'_{\Gamma_1}$, then for each $N$
$||t_1||_{N,C_\alpha,\psi_1}$ is uniformly
bounded.
The estimate of $I_3$ is finally
\begin{eqnarray*}
|I_3(\zeta)| & \le & p_{B}(t_2) (1+|\zeta|)^{ -N}.
\end{eqnarray*}
\end{proof}
For each $j$, the conditions of the lemma hold if we
put $\psi_1=\varphi_1\psi_{j1}$,
$\psi_2=\varphi_2\psi_{j2}$, 
$W_1=W_{j1}$ and $W_2=W_{j2}$. 
Thus,  for every bounded set $A$ in $\calD'_{\Gamma_1}$,
every $u\in A$ and every $v\in \calD'_{\Gamma_2}$ we have
\begin{eqnarray*}
||u\otimes v||_{N,V,\varphi_1\otimes\varphi_2} & \le &
\sum_j
||u\otimes v||_{N,V,\varphi_1\psi_{j1}\otimes
  \varphi_2\psi_{j2}} \\& \le & \sum_j
  M_{1j} 
   ||v||_{N,C_{\beta_j},\varphi_2\psi_{j2}}
+ 
 M_2
   ||v||_{N+m,C_{\beta_j},\varphi_2\psi_{j2}} 
+ p_{B_j}(v).
\end{eqnarray*}
Since the sum over $j$ is finite, this means that
the family of maps $u\times v \mapsto u\otimes v$,
where $u\in A$, is equicontinuous for any
bounded set $A\subset\calD'_{\Gamma_1}$.
Because of the symmetry of the problem, we can prove 
similarly that
the family of maps $u\times v \mapsto u\otimes v$,
where $v\in B$, is equicontinuous for any
bounded set $B\subset\calD'_{\Gamma_2}$.
Finally, we have proved

\begin{thm}\label{hypotens}
Let $\Omega_1\subset \bbR^{d_1}$,
$\Omega_2\subset \bbR^{d_2}$ be open sets,
$\Gamma_1\in \dotT^*\Omega_1$,
$\Gamma_2\in \dotT^*\Omega_2$ be closed cones and 
\begin{eqnarray*}
\Gamma &= & 
\big(\Gamma_1 \times \Gamma_2\big) \cup 
\big((\Omega_1\times\{0\})\times \Gamma_2\big)
\cup \big(\Gamma_1\times (\Omega_2\times\{0\})\big).
\end{eqnarray*}
Then, the tensor product 
$(u,v)\mapsto u\otimes v$
is hypocontinuous from
$\calD'_{\Gamma_1}\times\calD'_{\Gamma_2}$ to $\calD'_\Gamma$,
in the normal topology.
\end{thm}

\section{The pull-back}
\label{pullbacksect}
The purpose of this section is to prove
\begin{prop}\label{pullbackprop}
Let $\Omega_1\subset \mathbb{R}^{d_1}$ and $\Omega_2\subset
\mathbb{R}^{d_2}$
be two open sets and 
$\Gamma$ a closed cone in $\dotT^*\Omega_2$. 
Let $f:\Omega_1 \to \Omega_2$ be a smooth map such that
$N_f\cap \Gamma = \emptyset$, where
$N_f =  \{(f(x); \eta)\in \Omega_2\times \mathbb{R}^n\telque
  \eta\circ df_x =0\}$ and 
$f^*\Gamma = \{(x; \eta\circ df_x)\telque (f(x);\eta)\in \Gamma\}$,
where
\begin{eqnarray*}
 \eta\circ df_x &:=&  
\sum_{j=1}^{d_2} \eta_jd(y^j \circ f)_x  = 
\sum_{j=1}^{d_2} \eta_jdy^j \circ df_x\\ 
&=& 
 \sum_{j=1}^{d_2}  \eta_jdf^j_x 
 =\sum_{j=1}^{d_2} \sum_{i=1}^{d_1}  
   \eta_j \frac{\partial f^j}{\partial x^i} dx^i.
\end{eqnarray*}
Then, the pull-back operation
$f^*: \calD'_\Gamma(\Omega_2)\to \calD'_{f^*\Gamma}(\Omega_1)$
is continuous for the normal topology.
\end{prop}

We will show this by proving
that $\langle f^*u,v\rangle$ is continuous for 
every $v$ in an equicontinuous set.
Before doing so, we characterize the equicontinuous
sets of the normal topology, which is of independent
interest.

\subsection{Equicontinuous subsets}
Let $\Omega$ be open in $\bbR^d$ and
$\Gamma$ be a closed cone in $\dotT^*\Omega$.
We define the open cone
$\Lambda=\{(x,\xi)\in \dotT^*\Omega\telque
(x,-\xi)\notin \Gamma\}$ and the
space
$\calE'_\Lambda(\Omega)$ of compactly supported distributions
$v\in \calE'(\Omega)$ such that $\WF(v)\subset \Lambda$.

Acccording to
Definition (\ref{equicontdefpremiere}),
a set $H$ is equicontinuous
in $\calE'_\Lambda(\Omega)$ (which is the strong dual
of $\mathcal{D}^\prime_\Gamma(\Omega)$~\cite{Dabrouder-13}) 
if and only if there is a finite number of 
seminorms 
$||\cdot||_{N_1,V_1,\chi_1},\dots,||\cdot||_{N_k,V_k,\chi_k}$
of $\mathcal{D}^\prime_\Gamma(\Omega)$,
a bounded subset $B_0$ of $\calD(\Omega)$
and a constant $M$  such that
\begin{eqnarray}
|\langle u,v\rangle| & \le & M \sup\{||u||_{N_1,V_1,\chi_1},
\dots, ||u||_{N_k,V_k,\chi_k},p_{B_0}(u)\}
\label{defequicont}
\end{eqnarray}
for every $u\in \calD'_\Gamma(\Omega)$ and every $v\in H$.
There is only one seminorm $p_{B_0}$ because these
seminorms are saturated~\cite[p.~107]{Horvath}
in $\calD'(\Omega)$ with the strong topology.
The following theorem will be useful to
prove the continuity of linear maps~\cite[p.~200]{Horvath}:
\begin{thm}
If $E$ is a locally convex space and
$f: E \to \calD'_\Gamma(\Omega)$ is a linear map, then
$f$ is continuous if and only if, for every
equicontinuous set $H$ of $\calE'_\Lambda(\Omega)$
the seminorm $p_H:E\to\bbR$ defined by
$p_H(x) = \sup_{v\in H} |\langle f(x),v\rangle|$
is continuous.
\end{thm}
The equicontinuous sets of $\calE'_\Lambda(\Omega)$ are known:
\begin{lem}
\label{charequilem}
A subset $B$ of $\calE'_\Lambda(\Omega)$ 
is equicontinuous
if and only if there is: (i) a compact set $K\subset\Omega$ containing
the support of all elements of $B$; (ii) a closed
cone $\Xi\subset\Lambda$ such that 
$B\subset \calD'_\Xi(\Omega)$, 
$B$ is bounded in $\calD'_\Xi(\Omega)$ 
and
$\underline{\pi}(\Xi)\subset K$.
\end{lem}
\begin{proof}
We first prove that every such $B$ is equicontinuous.
We showed in~\cite{Dabrouder-13} that
the space $\calE'_\Lambda(\Omega)$ is the inductive limit
of spaces $E_\ell=\{v\in \calE'_\Lambda(\Omega)\telque
\supp v\in L_\ell, \WF(v)\in \Lambda_\ell\}$,
where the compact sets $L_\ell$ exhaust $\Omega$
and the closed cones $\Lambda_\ell$ exhaust $\Lambda$.
Thus, there is an integer $\ell$ such that
$\Xi \subset \Lambda_\ell$ and 
$B\subset E_\ell$. The inclusion of $\Xi$ in $\Lambda_\ell$
implies that every seminorm $||\cdot||_{N,V,\chi}$ of
$E_\ell$ is also a seminorm of $\calD'_\Xi(\Omega)$ 
because
$\supp\chi\times V$ does not meet
$\Xi$ if it does not meet $\Lambda_\ell$. 
Thus, $B$ is bounded in $E_\ell$ and 
Eq.~(8) of~\cite{Dabrouder-13} gives us
\begin{eqnarray*}
\sup_{v\in B}|\langle u,v\rangle| & \le & \sum_j
\Big(
p_{B_j}(u) + ||u||_{m+n+1,V_j,\chi_j} C
I_n^{n+1}
+ ||u||_{n,V_j,\chi_j} M_{n,W_j,\chi_j}
  I_n^{2n} \Big),
\end{eqnarray*}
which can be converted to the equicontinuity
condition \eqref{defequicont}.

To show the converse, we denote by $B$ the
set of all $v\in \calE'_\Lambda(\Omega)$ that satisfy
Eq.~\eqref{defequicont}.
Then, by following exactly the proof of
Prop.~7 of \cite{Dabrouder-13},
we obtain that the support of all
elements of $B$ is included in a compact set
$K=\cup_j \supp\chi_j \cup K'$, where 
$K'$ is a compact set containing the support
of all $f\in B_0$.  Moreover, 
the wave front set of all elements of
$B$ is contained in
$\Xi=\cup_j \supp\chi_j\times (-V_j)$.
It remains to show that $B$ is bounded in
$\calD'_\Xi(\Omega)$ for the normal
topology.
We first notice that, if 
$\supp\chi\times (-V)\subset \Xi$, then
$||\cdot||_{N,V,\chi}$ is a continuous
seminorm of the strong dual $\calE'_{(\Xi')^c}(\Omega)$
of $\calD'_\Xi(\Omega)$.
Indeed, it was shown in the proof
of Prop.~7 of \cite{Dabrouder-13} that
$||u||_{N,V,\chi}=\sup_{\xi\in V}|\langle u,f_\xi\rangle|$,
where $f_\xi(x)=(1+|\xi|)^N \chi(x) e^{ -i\xi\cdot x}$
and the set $\{f_\xi,\xi\in V\}$ is bounded in 
$\calD'_\Xi(\Omega)$.
If $B'$ is a bounded set in $\calE'_{(\Xi')^c}(\Omega)$,
the continuous seminorms $||u||_{N,V,\chi}$
and $p_{B_0}(u)$ of $\calE'_{(\Xi')^c}(\Omega)$
appearing on the
right hand side of (\eqref{defequicont})
are bounded over $B'$.
Thus, for any bounded set $B^\prime$ in
$\calE'_{(\Xi')^c}(\Omega)$, taking $u\in B^\prime$
and taking the sup in (\eqref{defequicont})
over $u\in B^\prime$ yields that  
$\sup_{u\in B',v\in B} |\langle u,v\rangle|$ is bounded
and $B$ is a bounded subset of
$\calD'_\Xi(\Omega)$ when $\calD'_\Xi(\Omega)$ is equipped
with the strong $\beta(\calD'_\Xi,\calE'_{(\Xi')^c})$ topology.
It is shown in~\cite[Theorem 33]{Dabrouder-13} that the bounded
sets of $\calD'_\Gamma(\Omega)$ coincide for the strong and
the normal topologies. Thus, $B$ is bounded for
the normal topology.
\end{proof}

We obtain the following characterization of continuous 
linear maps:
\begin{thm}
\label{charcont}
Let $E$ be a locally convex space, $\Omega$ an open subset
of $\bbR^d$ and $\Gamma$ a closed cone in $\dotT^*\Omega$.
A linear map $f:E\to \calD'_\Gamma(\Omega)$ is continuous if and only if every map
$f_B: E\to \bbR$ defined by 
$f_B(x)=\sup_{v\in B}|\langle f(x),v\rangle|$
is continuous, 
where $B$ is equicontinuous in 
$\mathcal{E}^\prime_\Lambda
(\Omega)$, with
$\Lambda=\left(\Gamma^\prime\right)^c$.
\end{thm}
The equicontinuous sets of $\calE'_\Lambda(\Omega)$ intervene
also because the duality pairing enjoys a 
sort of hypocontinuity where, for $\calE'_\Lambda(\Omega)$,
the bounded sets are replaced by the equicontinuous ones:
\begin{thm}\label{thmdualityequi}
Let the duality pairing  $\calD'_\Gamma(\Omega)\times
\calE'_\Lambda(\Omega)\to\bbK$
be defined by $u\times v\to f(u,v)=\langle u,v\rangle$.
Then, for every bounded set $A$ of $\calD'_\Gamma(\Omega)$ and every
equicontinuous set $B$ of $\calE'_\Lambda(\Omega)$ the sets of maps
$\{f_u\telque u\in A\}$ and 
$\{f_v\telque v\in B\}$ are equicontinuous~\cite[p.~157]{Kothe-II}.
\end{thm}

\subsection{Proof of continuity of the pull-back}

\paragraph{Strategy
of the proof.}
Let $\Omega_1$ and $\Omega_2$ be open sets in 
$\bbR^{d_1}$ and $\bbR^{d_2}$, respectively.
Let $f:\Omega_1\to\Omega_2$ be a smooth map and
$\Gamma$ be a closed cone in $\dotT^*\Omega_2$.
We want to show that the pull-back 
$f^*:\calD'_\Gamma(\Omega_2)
\to \calD'_{f^*\Gamma}(\Omega_1)$ is continuous
for the normal topology.
According to Theorem~\ref{charcont},
the pull-back is continuous if and only if,
for every equicontinuous set 
$B\subset \calE'_\Lambda(\Omega_1)$
(where $\Lambda=(f^*\Gamma)^{',c}$) the family of maps
$(\rho_v)_{v\in B}$, defined by
$\rho_v:u\mapsto \langle f^*u,v\rangle$,
is equicontinuous 
which implies that
$\sup_{v\in B}\vert 
\left\langle f^*u,v \right\rangle
\vert$
is continuous in
$u$. 
By Lemma~\ref{charequilem}, we know that there is a compact set
$K\subset \Omega_1$ and a closed cone $\Xi\subset(f^*\Gamma')^c$
such that $\supp v\subset K$ and $\WF(v)\subset \Xi$ for all
$v\in B$. Choose a function  $\chi\in \calD(\Omega_1)$ 
such that $\chi|_K=1$. If $(\varphi_i)_{i\in I}$ is a partition of unity
of $\Omega_2$, we can write
$\langle f^*u,v\rangle=\sum_i\langle f^*(u\varphi_i),v\chi\rangle$.
The image of $\supp\chi$ by 
$f$ being compact~\cite[p.~19]{Bredon}, only a finite
number of terms of this sum are nonzero and 
the family $\rho_v$ is equicontinuous if and only if,
for every $\varphi\in \calD(\Omega_2)$,
the family of maps $u\mapsto \langle f^*(u\varphi),v{\chi}\rangle$
is equicontinuous.
\paragraph{Stationary phase and Schwartz kernels.}

In order to calculate the pairing between $f^*(u\varphi)$ and $v$,
we first
notice that, 
when $u$ is a locally integrable function, then
$u\varphi(y)=\calF^{-1}(\widehat{u\varphi})(y)=
(2\pi)^{-d_2}\int_{\bbR^{d_2}} d\eta e^{i\eta\cdot y} 
\widehat{u\varphi}(\eta)$,
so that 
$f^*\left(u\varphi\right)(x)=(2\pi)^{-d_2} \int_{\bbR^{d_2}} d\eta 
  e^{i\eta\cdot f(x)} \widehat{u\varphi}(\eta)$
and
\begin{eqnarray*}
\langle f^*\left(u\varphi\right), \chi v\rangle &=& \frac{1}{(2\pi)^{d_2}}
\int_{\Omega_1} 
\int_{\bbR^{d_2}} 
 \chi(x) v(x)
e^{i\eta \cdot f(x)} \widehat{u\varphi}(\eta)
d x d \eta\\
&=&\frac{1}{(2\pi)^{d_2}}
\int_{\bbR^{d_2}}
\int_{\Omega_1} 
\int_{\bbR^{d_2}} 
 \chi(x) v(x)
e^{i\eta \cdot f(x)} e^{-iy\cdot\eta}u(y)\varphi(y)
dyd x d \eta .
\end{eqnarray*}
This definition can be extended to any distribution
$u\in \calD'_\Gamma$ as
\begin{eqnarray}
\label{fsuphiv}
\langle f^*(u\varphi), \chi v\rangle &=& \frac{1}{(2\pi)^{d}}
  \int_{\bbR^{d}} 
   u(y)v(x)I(x,y)dxdy,
\end{eqnarray}
where $d=d_1+d_2$, 
$I(x,y)=\int_{\bbR^{d_2}}
e^{i\eta \cdot( f(x)-y)} \varphi(y)\chi(x) d \eta$.
The duality pairing can also be written
$\langle f^*(u\varphi),v\chi\rangle = 
  \langle v\otimes u, I\rangle$.
Note that 
$I(x,y)=(2\pi)^{-d_2}
\chi(x)\varphi(y)
\int d\eta e^{i\eta\cdot( f(x)-y)}$
is an oscillatory integral~\cite{HormanderI} with
symbol $\chi(x)\varphi(y)$ and
phase $\eta\cdot( f(x)-y)$ where 
$\eta\cdot( f(x)-y)$ is
homogeneous of degree
$1$ with respect to $\eta$, for all $\eta\neq 0,
d\left(\eta\cdot( f(x)-y)\right)\neq 0$. 
Therefore, 
$I\in\mathcal{D}^\prime(\Omega_1\times\Omega_2)$ 
is the 
Schwartz
kernel of the bilinear
continuous map:
$(u,v) \mapsto \left\langle f^*(u\varphi), v\chi \right\rangle$.

\paragraph{ Proof of Proposition \ref{pullbackprop}.}

By Theorem~\ref{hypotens},
the map $(v,u)\mapsto v\otimes u$ is hypocontinuous
from $\calD'_\Xi\times \calD'_\Gamma$ to
$\calD'_{\Gamma_\otimes}$ where 
$\Gamma_\otimes=\Xi\times \Gamma \cup (\Omega_1\times\{0\})\times \Gamma
\cup \Xi\times(\Omega_2\times\{0\})$.
Let $\Lambda_{\otimes}$
be the open cone $\Gamma_{\otimes}^{\prime,c}$.
Therefore by Theorem \ref{thmdualityequi}, the 
family of duality pairings
$u\otimes v\in\calD'_{\Gamma_\otimes} \mapsto 
\left\langle u\otimes v,r \right\rangle$
is equicontinuous
from $\calD'_{\Gamma_\otimes}$ to $\bbK$
uniformly in $r\in B^\prime$
for every equicontinuous set $B'$ of $\calE'_{\Lambda_\otimes}$.
In particular, if $B'$ contains only the element
$I$, then 
the map $v\otimes u\mapsto \langle v\otimes u,I\rangle$
would be continuous  
since
$I$ is compactly supported in 
$\supp\chi\times\supp\varphi$ and its
wave front set is contained in $\Lambda_\otimes$.
Thus, if $\WF(I)\subset \Lambda_\otimes$,
then the map $(v,u)\mapsto \langle v\otimes u,I\rangle$
is hypocontinuous by the next lemma.
\begin{lem}
\label{hypocontcont}
The composition of a hypocontinuous map by a continuous linear map is
hypocontinuous.
\end{lem}
\begin{proof}
Let $f:E\times F\to G$ be a hypocontinuous map
and $g:G\to H$ a continuous linear map.
The map $g\circ f$ is hypocontinuous if and only if, for every bounded set
$B\subset F$
and every neighborhood $W$ of zero in $H$, there is a neighborhood
$U$ of zero in $E$ such that $(g\circ f)(U\times B)\subset W$
(with the similar condition for $(g \circ f)(A\times
V)$).
By the continuity of $g$, there is a neighborhood $Z$ of zero in
$G$ such that $g(Z)\subset W$. By the hypocontinuity of $f$,
there is a neighborhood $U$ of zero in $E$ such that
$f(U\times B)\subset Z$. Thus, $(g\circ f)(U\times B)
\subset g(Z)\subset W$.
\end{proof}
Therefore, the map
$(v,u)\mapsto \langle f^*(u\varphi),\chi v\rangle$ is hypocontinuous,
by item (i) of Definition~\ref{hypodef}, 
this implies that the family of maps
$\rho_v:u\mapsto \langle f^*(u\varphi),\chi v\rangle$
with $v\in B$ is equicontinuous.
It just remains to check that
$\WF(I)\subset \Lambda_\otimes$, i.e. that
$\WF(I)'$ does not meet $\Gamma_\otimes$.
The wave front set of $I$ is 
$\WF(I)\subset \{(x,f(x);-\eta\circ d_x f,\eta)\telque 
x\in \supp\chi\}$~\cite[p.~260]{HormanderI}.
Recall that $\Xi\subset(f^*\Gamma')^c=\{(x,-\eta\circ df_x)\telque
  (f(x),\eta)\notin\Gamma\}$. By definition of
$\Gamma_\otimes$ we must satisfy the following three conditions:
\begin{itemize}
\item $\Xi\times \Gamma \cap \WF(I)'=\emptyset$ because
it is the set of points 
$(x,f(x);-\eta\circ d_x f,\eta)$ such that
$(f(x),\eta)\notin\Gamma$ by definition of $\Xi$
and $(f(x),\eta)\in \Gamma$ by definition of $\Gamma$;
\item $\Xi\times (\Omega_2\times\{0\}) \cap \WF(I)'=\emptyset$ because
we would need $\eta=0$ whereas $(y,\eta)\in \Gamma$ implies
$\eta\not=0$; 
\item
$(\supp\chi\times\{0\})\times \Gamma\cap \WF(I)'\subset
\{(x,f(x);0,\eta)\telque x\in\supp\chi, \eta\circ df_x=0,
(f(x),\eta)\in \Gamma\}$.
\end{itemize}
Thus, if $f^*\Gamma\cap N_f=\emptyset$, 
then $\WF(I)'\cap\Gamma_\otimes=\emptyset$ and
the pull-back is continuous.

\paragraph{How to write the pull--back operator in terms of the Schwartz kernel $I$ ? Relationship with the product of distributions.} \label{kernelfonctoriel}
We start from a linear operator 
$L:\mathcal{D}(\mathbb{R}^{d_2})\longmapsto \mathcal{D}^\prime(\mathbb{R}^{d_1}) $ 
with corresponding Schwartz kernel 
$K\in\mathcal{D}^\prime(\mathbb{R}^{d_1}\times\mathbb{R}^{d_2})$.
Using the standard
operations on distributions, we
can make sense of the well--known
representation formula
$Lu=\int_{\mathbb{R}^{d_2}}  K(x,y) u(y)dy $
for an operator $L:\mathcal{D}(\mathbb{R}^{d_2})\longmapsto \mathcal{D}^\prime(\mathbb{R}^{d_1})$
and its kernel $K\in\mathcal{D}^\prime(\mathbb{R}^{d_1}\times\mathbb{R}^{d_2})$.
Let us define the two projections
$\pi_2:=(x,y)\in\mathbb{R}^{d_1}\times \mathbb{R}^{d_2}\longmapsto y\in\mathbb{R}^{d_2}$
and
$\pi_1:=(x,y)\in\mathbb{R}^{d_1}\times \mathbb{R}^{d_2}\longmapsto x\in\mathbb{R}^{d_1}$,
then we define
$K(x,y)u(y)=K(x,y)\pi_2^*u(x,y)=K(x,y)\left(1(x)\otimes u(y)\right)$ 
where $\pi_2^*u=1\otimes u$ 
and $\int_{\mathbb{R}^{d_2}}dyf(x,y)=\pi_{1*}f(x)$.
Therefore 
\begin{equation}
Lu=\int_{\mathbb{R}^{d_2}}  K(x,y) u(y)dy=\pi_{1*}\left(K\left(\pi_2^* u\right) \right).
\end{equation}

The interest
of the formula $Lu=\pi_{1*}\left(K\left(\pi_2^* u\right) \right)$ 
is that everything 
generalizes
to oriented manifolds. Replace
$\mathbb{R}^{d_2}$ (resp $\mathbb{R}^{d_1}$) 
with a 
manifold $M_2$ (resp $M_1$) 
with smooth volume densities 
$\vert\omega_2\vert$ 
(resp $\vert\omega_1\vert$), 
the duality pairing is
defined as the extension
of the usual 
integration 
against the 
volume densities, for
instance:
$\forall (u,\varphi)\in C^\infty(M_1)\times \mathcal{D}(M_1),
\left\langle u,\varphi\right\rangle_{M_1}=\int_{M_1}(u\varphi)\omega_1 $.
Finally, for the 
linear continuous map $L:=u\in\mathcal{D}(\mathbb{R}^{d_2})\longmapsto \chi f^*(u\varphi)\in \mathcal{D}^\prime(\mathbb{R}^{d_1})$, we get the formula:
$Lu= \pi_{1*}\left( I (\pi_2^* u)\right)$
where $I(x,y)=(2\pi)^{-d_2}
\chi(x)\varphi(y)
\int d\eta e^{i\eta\cdot(f(x)-y)}$ is the Schwartz kernel of $L$.

\subsection{Pull-back by families of smooth maps}
To renormalize quantum field theory in curved spacetimes,
it will be crucial to pull-back by family of smooth maps.
We start with a simple lemma.
\begin{lem}\label{Gammaclosed}
Let $\Omega_1,\Omega_2,U$
be open sets
in
$\mathbb{R}^{d_1},\mathbb{R}^{d_2},\mathbb{R}^{n}$
respectively.
For any compact sets
$(K_1\subset \Omega_1,K_2\subset\Omega_2,A\subset U)$
and
$f$ a smooth map
$f:\Omega_1\times U \to\Omega_2$,
the conic set
\begin{eqnarray*}
\Gamma
=\{(x,f(x,a);-\eta\circ d_xf(x,a),\eta)\telque (x,a,f(x,a))\in K_1\times A\times K_2,\eta\neq 0 \}
\end{eqnarray*}
is closed 
in $\dot{T}^*\left(\Omega_1\times\Omega_2\right)$.
\end{lem}
\begin{proof}
Let $(x,y;\xi,\eta)\in\overline{\Gamma}$
such that $(\xi,\eta)\neq (0,0)$. 
Then there is
a sequence 
\begin{eqnarray*}
(x_n,f(x_n,a_n);-\eta_n\circ d_xf(x_n,a_n),\eta_n)\in\Gamma, 
(x_n,a_n,f(x_n,a_n))\in K_1\times A\times K_2
\end{eqnarray*}
which converges to $(x,y;\xi,\eta)$.
By compactness of $A$, we extract
a convergent subsequence $a_n\rightarrow a$. By continuity
of $d_xf$,
we find that
$\xi=-\eta\circ d_xf(x,a)$,
we also find that
$\underset{n\rightarrow \infty}{\lim} 
f(x_n,a_n)=f(x,a)\in K_2$
since $K_2$ is closed
and we finally 
note that
we must have 
$\eta\neq 0$
otherwise $\xi=0,\eta=0$.
Therefore 
$(x,y;\xi,\eta)\in\Gamma$
by definition.
Finally, 
$\overline{\Gamma}
\subset\Gamma$
hence $\Gamma$ is closed. 
\end{proof}
\begin{prop}\label{Ifaequicontinuous}
Let $\Omega_1$ be an open set in $\bbR^{d_1}$,
$A\subset U\subset \bbR^n$ 
where $A$ is compact, $U$ 
and $\Omega_2$ are open sets 
in $\bbR^{d_2}$.
Let $\chi\in \calD(\Omega_1)$, 
$\varphi\in \calD(\Omega_2)$ and
$f$ a smooth map
$f:\Omega_1\times U \to\Omega_2$.
\begin{enumerate}
\item Then
the family of distributions $(I_{f(.,a)})_{a\in A}$ 
formally defined by
\begin{eqnarray*}
I_{f(.,a)}(x,y) &=& 
 \chi(x)\varphi(y)
  \int_{\bbR^{d_2}} \frac{d\theta}{(2\pi)^{d_2}}
   e^{i\theta\cdot( f(x,a)-y)}
\end{eqnarray*}
is a bounded set in $\calD'_{\Gamma}$, 
where
$\Gamma$ 
is the 
closed cone in $\dotT^*(\Omega_1\times\Omega_2)$
defined by:
\begin{eqnarray*}
\Gamma &=&\{(x,f(x,a);-\eta\circ d_xf(x,a),\eta)\telque
\\ & &  x\in \supp\chi,
f(x,a)\in \supp\varphi,a\in A,\eta\neq 0 \}.
\end{eqnarray*}
\item For any
open cone $\Lambda$ 
containing $\Gamma$,
$(I_{f(.,a)})_{a\in A}$
is
equicontinuous in
$\mathcal{E}^\prime_{\Lambda}(\Omega_1\times\Omega_2)$.
\end{enumerate}
\end{prop}

We will use the 
pushforward Theorem
\ref{pushforwardthm},
whose proof will be given in
Section $7$,
in the following proof.

\begin{proof}
 From Lemma
\ref{charequilem} and from the fact that
$(I_{f(.,a)})_{a\in A}$
is supported
in a fixed compact set 
$\supp\chi\times\supp\varphi$, we deduce that
conclusion (2) follows from the first claim
thus it suffices
to prove the
claim (1).

The conic set $\Gamma$ is closed
by Lemma \ref{Gammaclosed}. 
To prove
that the family
$\left(I_{f(.,a)}\right)_{a\in A}$
is bounded in $\mathcal{D}^\prime_\Gamma$,
it suffices
to check that
$\forall v\in \mathcal{E}^\prime_{\Gamma^{\prime
c}}(\Omega_1\times\Omega_2)$,
$\sup_{a\in A}\vert\left\langle I_{f(.,a)},v \right\rangle\vert  <
+\infty$
because of 
\cite[Proposition 1]{Dabrouder-13}.

\textbf{Step 1}
Our goal is to study
the map
$a\longmapsto \int_{\Omega_1\times\Omega_2} I_f(x,y,a)v(x,y)$
where 
\begin{eqnarray*}
I_{f}(x,y,a) = 
 \chi(x)\varphi(y)
  \int_{\bbR^{d_2}} \frac{d\theta}{(2\pi)^{d_2}}
   e^{i\theta\cdot(f(x,a)-y)}
\end{eqnarray*}
Let 
$\pi_{12},\pi_3$
be projections from $\Omega_1\times\Omega_2\times U$
defined by 
$\pi_{12}(x,y,a)=(x,y)$ and 
$\pi_3(x,y,a)=a$.
Using the dictionary explained 
in paragraph
\ref{kernelfonctoriel},
if $v$ were a 
test function,
then we would
find
that
\begin{eqnarray}\label{pushforwardkeyformula}
\int_{\Omega_1\times\Omega_2} I_f(x,y,.)v(x,y)
dxdy &=& \pi_{3*}\left(I_f\pi_{12}^*v \right)\in \mathcal{D}^\prime(U).
\end{eqnarray}
We want to prove 
that $a\longmapsto  
\int_{\Omega_1\times\Omega_2} I_f(x,y,a)v(x,y)dxdy$
is smooth 
in some open
neighborhood
of $A$ since this
would imply
that 
\begin{eqnarray*}
\sup_{a\in A}\vert 
\int_{\Omega_1\times\Omega_2} I_f(x,y,a)v(x,y)dxdy\vert
=\sup_{a\in A}\vert\left\langle
I_{f(.,a)},v\right\rangle\vert<+\infty.
\end{eqnarray*}
In order to do so, it suffices
to prove that the condition
$v\in \mathcal{E}^\prime_{\Gamma^{\prime,c}}$
implies that
the distributional 
product $I_f(x,y,a)v(x,y)=I_f\left(\pi_{12}^*v \right)(x,y,a)$
makes sense in $\mathcal{D}^\prime(\Omega_1\times \Omega_2\times U)$
and the push--forward 
$\pi_{3*}\left(I_f\pi_{12}^*v \right)= \int_{\Omega_1\times\Omega_2} I_f(x,y,.)v(x,y)dxdy$
has empty wave front set over
some open neighborhood
of $A$.

\textbf{Step 2}
The wave front set $WF(I_f)$ is the set of all 
\begin{eqnarray*}
\left(
\begin{array}{ccc}
x&;&-\theta\circ d_xf\\
f(x,a)&;&\theta\\
a&;&-\theta\circ d_af  
\end{array} \right)  
\end{eqnarray*}
such that $x\in \text{supp }\varphi,f(x,a)\in\text{supp }\chi, a\in U , \theta\neq 0$. 
And the wave front set $WF(\pi_{12}^*v)$ is the set of all
$\left(\begin{array}{ccc}
x&;&\xi\\
y&;&\eta\\
a&;&0 \end{array} \right) $
such that
$\left(\begin{array}{ccc}
x&;&\xi\\
y&;&\eta
\end{array}\right)\in WF(v).$

One also have 
$v\in\mathcal{E}^\prime_{\Gamma^{\prime c}}$ implies
$WF(v)\cap \Gamma^\prime=\emptyset$ so that
$\xi\neq -\eta\circ d_xf$ and
\begin{eqnarray*}
\forall \theta , \left(\begin{array}{c}
\xi-\theta\circ d_xf \\
\theta+\eta
\end{array}\right)=
\left(\begin{array}{c}
0 \\
0
\end{array}\right)
\text{has no solution}
\end{eqnarray*}
Observe that
\begin{eqnarray*}
WF(I_f)+WF(\pi_{12}^*v)=\{ \left(\begin{array}{ccc}
x&;& \xi-\theta\circ d_xf \\
f(x)&;&\theta+\eta \\
a&;& -\theta\circ d_af
\end{array}\right),\forall\theta \in\mathbb{R}^d\setminus \{0\} \},
\end{eqnarray*}
implies
$\left(WF(I_f)+WF(\pi_{12}^*v)\right)
\cap \underline{0}=\emptyset$ and 
$\left(WF(I_f)\cup WF(\pi_{12}^*v)\right)
\cap \underline{0}=\emptyset$.

\textbf{Step 3} In the last step,
we shall prove that the condition
$WF(v)\cap \Gamma^\prime=\emptyset$
actually implies that
$WF\left(\pi_{3*}\left(I_f\pi_{12}^*v\right)\right)$
is empty over some
open neighborhood $U^\prime$
of $A$. 
The condition
$
WF(v)\cap\Gamma^\prime=\emptyset 
$
implies the existence
of
some open 
neighborhood $U^\prime$
of $A$ s.t.
\begin{eqnarray*}
\forall a\in U^\prime, \forall (x,f(x,a);\xi,\eta)\in WF(v),\xi\neq -\eta\circ d_xf(x,a).
\end{eqnarray*}
Since $A$ and $\text{supp }v$ are compact
and $A\times WF(v)$ is closed,
we can find $\delta>0$ s.t. $$\forall (a,(x,f(x,a);\xi,\eta))\in A\times WF(v),
\,\ \vert\xi+\eta\circ d_xf(x,a)\vert\geqslant \delta\vert\eta\vert.$$
Define $U^\prime=\{a\in U \telque  \forall (x,f(x,a);\xi,\eta)\in WF(v),\vert\xi+\eta\circ d_xf(x,a)\vert > \frac{\delta}{2}\vert\eta\vert\}$.
Therefore, the
condition 
$WF(v)\cap\Gamma^\prime=\emptyset$
on the wave front set of
$v$
ensures that
$\pi_{3*}\left(I_f\pi_{12}^*v\right)$
is well defined in
$\mathcal{D}^\prime_\emptyset(U^\prime)=C^\infty(U^\prime)$.
Hence $a\longmapsto \left\langle v,I_{f(.,a)}\right\rangle
= \int_{\Omega_1\times\Omega_2} I_f(x,y,a)v(x,y)dxdy$
is smooth on $U^\prime$, a fortiori 
continuous on the compact set $A$
which means that
$\sup_{a\in A}\vert \left\langle v,I_{f(.,a)}\right\rangle\vert <+\infty$.
\end{proof}

\begin{thm}\label{pbenfamille}
Let $\Omega_1\subset \mathbb{R}^{d_1},\Omega_2\subset
\mathbb{R}^{d_2}$ 
be two open sets, $A\subset U\subset \bbR^n$ where
$A$ compact, $U$
open 
and 
$\Gamma$ a closed cone in $\dotT^*\Omega_2$. 
Let $f:\Omega_1\times U \to \Omega_2$ be a smooth map such that
$\forall a\in A,  f(.,a)^*\Gamma$
does not meet 
the zero section
$\underline{0}$ 
and set 
$\Theta=\bigcup_{a\in A}f(.,a)^*\Gamma$.
Then for all seminorms 
$P_B$ of $\mathcal{D}^\prime_\Theta(\Omega_1)$,
$\forall u\in\mathcal{D}^\prime_\Gamma(\Omega_2)$, 
the family $P_B(f(.,a)^*u)_{a\in A}$ is bounded.
\end{thm}

\begin{proof}
We need to prove that
$\sup_{(v,a)\in B\times A} \vert  \langle f(.,a)^*u, v\rangle\vert<+\infty$
for any equicontinuous subset $B$ of 
$\mathcal{E}^\prime_{\Theta^{\prime,c}}(\Omega_1)$.
It follows from
Lemma \ref{charequilem}
that there exists 
some closed cone $\Xi$
such that $\Xi\cap\Theta^\prime=\emptyset$
and $B\subset \mathcal{D}_{\Xi}^\prime(\Omega_1)$.
Set $\Gamma_\otimes=\left(\Xi\times \Gamma\right) \cup \left((\Omega_1\times\{0\})\times \Gamma\right)
\cup \left(\Xi\times(\Omega_2\times\{0\})\right)$. 
Let $\Lambda_\otimes$ be the open cone
defined as $\Lambda_\otimes=(\Gamma_\otimes)^{\prime,c}$.
By proposition \ref{Ifaequicontinuous},
we can easily verify
as in the proof
of Proposition
\ref{pullbackprop} that
the 
family 
$(I_{f(.,a)})_{a\in A}$
is equicontinuous
in $\calE'_{\Lambda_\otimes}$.
Then we prove similarly
as in the proof
of the pull-back Proposition  
\ref{pullbackprop}
that the family of maps
$\rho_{v,a}:u\mapsto \langle f(.,a)^*(u\varphi),\chi v\rangle$
with $(v,a)\in B\times A$ is equicontinuous where
$B$ is equicontinuous
in $\mathcal{E}^\prime_{\Theta^{\prime,c}}(\Omega_1)$ and
$\chi$ is chosen in such a way that 
$\chi|_{\text{supp }B}=1$
and $\varphi|_{f(\text{supp }B)}=1$, therefore:
\begin{eqnarray*}
\sup_{(v,a)\in B\times A} \vert  \langle f(.,a)^*u, v\rangle\vert=
\sup_{(v,a)\in B\times A} \vert \langle f(.,a)^*(u\varphi),\chi v\rangle\vert<+\infty.
\end{eqnarray*}
and the result follows
from 
Theorem \ref{charcont}.
\end{proof}

\section{Product and push--forward of distributions}
\label{prodpushsect}
H\"ormander noticed that the product of distributions $u$ and $v$ can be
described
as the composition of the tensor product $(u,v)\mapsto u\otimes v$
with the pull-back by the map $f: x\mapsto (x,x)$.
If the wave front sets of $u$ and $v$ are contained in $\Gamma_1$ and
$\Gamma_2$, then
the wave front set of $u\otimes v$ is contained in
$\Gamma_\otimes=\left(\Gamma_1\times \Gamma_2 \right)\cup\left( (\Omega_1\times\{0\})\times
\Gamma_2\right)
 \cup \left( \Gamma_1 \times (\Omega_2\times\{0\})\right)$
and the pull-back is well-defined if
the set $N_f=\{ (x,x;\eta_1,\eta_2)\telque 
(\eta_1+\eta_2)  \circ  dx)=0 \}$, 
which is the conormal bundle 
of the diagonal 
$\Delta\subset \mathbb{R}^n\times\mathbb{R}^n$,
does not meet $\Gamma$, i.e. if there is no point $(x;\eta)$ in
$\Gamma_1$
such that $(x;-\eta)$ is in $\Gamma_2$.
This gives us
\begin{thm}
\label{hypocontprod}
Let $\Omega\subset \bbR^n$ be an open set
and $\Gamma_1,\Gamma_2$ be two closed cones in
$\dotT^*\Omega$ such that $\Gamma_1\cap \Gamma_2'=\emptyset$.
Then the product of distributions
is hypocontinuous for the normal topology from
$\calD'_{\Gamma_1}\times\calD'_{\Gamma_2}$
to $\calD'_\Gamma$, where
\begin{eqnarray}
\Gamma &= &
\big(\Gamma_1 \times_\Omega \Gamma_2\big) \cup
\big((\Omega\times\{0\})\times_\Omega \Gamma_2\big)
\cup \big(\Gamma_1\times_\Omega (\Omega\times\{0\})\big).
\label{Gammaprod}
\end{eqnarray}
\end{thm}
\begin{proof}
The product of distribution is the composition of the hypocontinuous
tensor product
with the continuous pull-back  (see Lemma~\ref{hypocontcont}).
\end{proof}

This theorem has the useful corollary:
\begin{cor}\label{prodsmooth}
Let $\Omega\subset \bbR^n$ be an open set
and $\Gamma$ be a closed cone in
$\dotT^*\Omega$.
Then the product of a smooth map and a distribution
is hypocontinuous for the normal topology from
$C^\infty(\Omega)\times \calD'_\Gamma$ to $\calD'_\Gamma$.
\end{cor}
\begin{proof}
By Lemma~\ref{topequlem},
$C^\infty(\Omega)$ and $\calD'_\emptyset$ are
topologically isomorphic. Therefore, the 
corollary follows by applying 
Theorem~\ref{hypocontprod} to $\Gamma_1=\emptyset$
and $\Gamma_2=\Gamma$. 
Equation \eqref{Gammaprod} shows that the wave front
set of the product is in $\Gamma$.
\end{proof}

\subsection{The push--forward as a consequence of the pull--back theorem.}
\begin{thm}\label{pushforwardthm}
Let $\Omega_1\subset \mathbb{R}^{d_1}$ and $\Omega_2\subset \mathbb{R}^{d_2}$
be two open sets and $\Gamma$ a closed cone
in $\dotT^*\Omega_1$. For any smooth map $f:\Omega_1
\to \Omega_2$ and any closed subset 
$C$ of $\Omega_1$ such that 
$f|_C: C\to \Omega_2$ is proper and $ \underline{\pi}(\Gamma)\subset C$,
then $f_*$ is continuous in the normal topology 
from
$\{u\in \calD'_\Gamma\telque \supp u \subset 
C\}$ to $\calD'_{f_*{\Gamma}}$,
where 
$f_*\Gamma=\{(y;\eta)\in \dotT^*\Omega_2\telque
\exists x\in \Omega_1\text{ with } y=f(x)\text{ and }
(x;\eta\circ df_x)\in \Gamma \cup \supp u\times \{0\}\}$.
\end{thm}
\begin{proof}
The idea of the proof is to think of a push--forward as the adjoint of 
a pull--back~\cite{Duistermaat}.
For all $B$ equicontinuous
in $\mathcal{E}^\prime_{\Lambda}$
where $\Lambda=f_*\Gamma^{\prime,c}$,
for all $v\in B, \text{supp }\left(f^* v\right)\cap C$
is contained in a fixed compact set $K$ 
since $f$ is proper on $C$. 
Hence, for any $\chi\in\mathcal{D}(\mathbb{R}^d)$ such that $\chi|_C=1$
and $f$ is proper on $\text{supp }\chi$, 
we should have at least formally
$\left\langle f_*u, v \right\rangle=\left\langle  u,\chi f^* v \right\rangle$
if the duality pairings make sense.
On the one hand, 
if $v\in\mathcal{E}^\prime_\Lambda(\Omega_2)$
where
$f^*\Lambda$ 
does not meet the zero section 
$\underline{0}\subset T^*\Omega_1$
then
the pull--back $f^*v$ 
would be well defined by the pull--back 
Proposition 
(\ref{pullbackprop})
(which is equivalent to the fact that 
$N_f\cap \Lambda=\emptyset$).
On the other hand, the duality pairing 
$\left\langle u , \chi(f^* v) \right\rangle$ 
is well defined if 
$f^*\Lambda\cap \Gamma^\prime=\emptyset$. 
Combining
both conditions leads to the
requirement that 
$f^*\Lambda\cap \left(\Gamma^\prime\cup \underline{0}\right)=\emptyset$.
But 
note that:
\begin{eqnarray*}
&&\left(f^*\Lambda\right)\cap 
\left(\Gamma^\prime\cup \underline{0}\right)=\emptyset \\
&\Leftrightarrow & 
\{(x;\eta\circ df) | (f(x);\eta)\in\Lambda,(x;\eta\circ df)\in \Gamma^\prime\cup\{0\}\}=\emptyset\\
&\Leftrightarrow &
 (f(x);\eta)\in\Lambda\implies (x;\eta\circ df)\notin \Gamma^\prime\cup\{0\}\\
 &\Leftrightarrow &
(f(x);\eta)\in\Lambda^\prime\implies (x;\eta\circ df)\notin \Gamma\cup\{0\}.
\end{eqnarray*}
which
is equivalent to the fact that $\Lambda^\prime$ 
does not meet
$f_*\Gamma=\{(f(x);\eta)\telque (x;\eta\circ df) 
\in \left(\Gamma\cup \underline{0}\right),\eta\neq 0 \}
\subset T^*\Omega_2$
which is exactly the
assumption of our theorem.
Therefore, the set of distributions
$\chi f^* B$ is supported in a fixed compact
set, 
bounded in 
$\mathcal{D}^\prime_{f^*\Lambda}(\Omega_1)$
by the pull--back Proposition \ref{pullbackprop} 
applied to $f^*$,
the duality pairings are well defined and:
$\sup_{v\in B}\vert \left\langle f_* u, v \right\rangle\vert=
\sup_{v\in B^\prime}\vert \left\langle u, v \right\rangle\vert $
where $B^\prime=\chi f^* B$ is equicontinuous in $\mathcal{E}^\prime_{\Gamma^{\prime,c}}(\Omega_1)$ (the support of the distribution is compact because $f$
is proper) which means
that $\sup_{v\in B^\prime}\vert \left\langle u, v \right\rangle\vert $ is 
a continuous seminorm
for 
the normal topology of 
$\mathcal{D}_\Gamma^\prime(\Omega_1)$.
\end{proof}
We state and prove
a parameter version
of the push--forward theorem
\begin{thm}\label{pshfwdenfamille}
Let $\Omega_1\subset \mathbb{R}^{d_1}$ 
and $\Omega_2\subset \mathbb{R}^{d_2}$
be two open sets, $A\subset U\subset \mathbb{R}^n$ 
where $A$
is
compact, $U$ is open and $\Gamma$ a closed cone
in $\dotT^*\Omega_1$. For any smooth map $f:\Omega_1\times U
\to \Omega_2$ and any closed subset 
$C$ of $\Omega_1$ such that 
$f: C\times A\to \Omega_2$ is proper 
and $\underline{\pi}(\Gamma)\subset C$,
then $f(.,a)_{*}$ is uniformly continuous  in the normal topology 
from
$\{u\in \calD'_\Gamma\telque \supp u \subset 
C\}$ to $\calD'_{\Xi}$,
where $\Xi=\cup_{a\in A }f(.,a)_{*}\Gamma$.
\end{thm} 
\begin{proof}
We have to check that 
$\Xi$ is closed over 
$\dot{T}^*_{f(C\times A)}\Omega_2$.
Let $(y;\eta)\in \overline{\Xi}
\cap \dot{T}^*_{f(C\times A)}\Omega_2$
then there exists a 
sequence $(y_n;\eta_n)\rightarrow (y;\eta)$
such that 
$(y_n;\eta_n)\in\Xi\cap 
\dot{T}^*_{f(C\times A)}\Omega_2$.
By definition,
$y_n=f(x_n,a_n)$
where 
$(x_n,\eta_n\circ d_x f(x_n,a_n))
\in\Gamma\cup \underline{0},
(x_n,a_n)\in C\times A$.
The central observation is that
$\overline{\{y_n| n\in\mathbb{N} \}}
 \subset f(C\times A)  $
and 
$\overline{\{(x_n,a_n)| f(x_n,a_n)=y_n,n\in\mathbb{N} \}}
\subset C\times A$
are compact sets because
$f$ is proper on $C\times A$.
Then we can extract 
convergent subsequences
$(x_n,a_n)\rightarrow (x,a)$
and $(x,\eta\circ d_x f(x,a))
\in\Gamma\cup \underline{0}$
since $\Gamma\cup \underline{0}$
is closed in $\dot{T}^*\Omega_2$
and $d_xf$
is continuous. 
By
definition of $\Xi$, this proves
that $(y;\eta)\in \Xi$
and we can 
conclude that
$\Xi$ is closed.
Now we can repeat 
the proof of the 
push--forward theorem
except
that we use 
the pull--back 
theorem
with parameters.
Let $B$ be 
equicontinuous
in $\mathcal{E}_{\Xi^{\prime,c}}^\prime(\Omega_2)$
hence all elements of $B$ have
support contained in
some compact.
We have
$\forall v\in B, \text{supp }(f^*v)\cap C\times A$
is compact, therefore
for $\chi=1$ on $\cup_{a\in A} \text{supp }(f(.,a)^*v)\cap C$,
the family $B^\prime=\{(\chi f(.,a)^*v)| a\in A,v\in B \}$
is equicontinuous in 
$\mathcal{E}^\prime_{\Theta}$
where $\Theta=\Gamma^{\prime,c}$
by the
parameter version of
the pull--back theorem.
Therefore,
$u\mapsto \sup_{a\in A} 
\sup_{v\in B}\vert \left\langle f(.,a)_{*} u, v \right\rangle\vert=\sup_{v\in B^\prime}\vert \left\langle u, v \right\rangle\vert$
is continuous in $u$
since the right hand term
is a 
continuous seminorm
for the normal
topology of
$\mathcal{D}^\prime_{\Gamma}(\Omega_1)$. 
\end{proof}

\paragraph{Convolution product.}
In the same spirit as for the multiplication
of distributions, 
the convolution product $u*v$ can be described as the
composition of the
tensor product $(u,v)\mapsto u\otimes v$
with the push--forward by the
map $\Sigma:=(x,y) \mapsto (x+y)$. 
For a closed subset $X\subset \mathbb{R}^n$,
let $\mathcal{D}_\Gamma^\prime(X)$ be the set
of distributions supported in $X$ with wave front
in $\Gamma$.
Therefore, we have the 
\begin{thm}\label{convolprod}
Let  
$\Gamma_1,\Gamma_2$ be
two closed conic sets in $\dot{T}^*\mathbb{R}^n$ and 
$X_1,X_2$
two closed subsets of $\mathbb{R}^n$
such that $\Sigma:X_1\times X_2 \mapsto \mathbb{R}^n$
is proper.
Then the convolution product
of distributions is hypocontinuous
from\\ $\mathcal{D}^\prime_{\Gamma_1}(X_1)\times\mathcal{D}^\prime_{\Gamma_2}(X_2)$
to $\mathcal{D}^\prime_\Gamma(X_1+X_2)$
where
\begin{eqnarray}
\Gamma=\{(x+y;\eta) \telque (x;\eta)\in\Gamma_1,(y;\eta)\in\Gamma_2 \}.
\end{eqnarray}
\end{thm}
\begin{proof}
The convolution product of distribution is the composition of the hypocontinuous
tensor product with the continuous push--forward.
\end{proof}

As an application of the parameter version 
of the push-forward theorem, we can state
the coordinate invariant definition of the
wavefront set, which was proposed by
Duistermaat~\cite[p.~13]{Duistermaat}, correcting
a first attempt by Gabor~\cite{Gabor-72}.
Its proof is left to the reader.
\begin{thm}\label{duistermaatWF}
Let $\Omega\subset\mathbb{R}^d$ be an open set, 
$u\in\mathcal{D}^\prime(\Omega)$.
An element $(x;\xi)\notin WF_{D}(u)$ if and only if
for all $f\in C^\infty(\Omega\times \mathbb{R}^n)$ 
such that
$d_xf(x,a_0)=\xi$ for some
$a_0\in \mathbb{R}^n$, there exists 
some neighborhoods
$A$ of $a_0$ and $U$ of $x_0$
such that for all
$\varphi\in\mathcal{D}(U)$:
$\vert\left\langle u,\varphi e^{i\tau f(.,a)} \right\rangle\vert
=O(\tau^{-\infty})$
uniformly in some 
neighborhood
of $a_0$ in $A$.
\end{thm}

\section{Appendix: Technical results}
This appendix gather different useful results.
Several of them are folkore results for which we could
find no proof in the literature.

\subsection{Exhaustion of the complement of $\Gamma$}
\label{exhaustionsubsec}
To prove that $\mathcal{D}^\prime_\Gamma$
is nuclear~\cite{Dabrouder-13}, we need to take the additional seminorms
in a countable set:
the complement $\Gamma^c=\dotT^*M\backslash\Gamma$
of any closed cone $\Gamma\subset \dotT^*M$
can be exhausted by a countable set of products
$U\times V$, where $U\subset M$ is compact and $V$ is a
closed conic subset
of $\bbR^n$.

First we introduce the sphere bundle over $M$ (or unit cotangent bundle)
$UT^*M=\{(x;k)\in T^*M \telque |k|=1\}$.
We then define the set 
$U\Gamma_K=\{(x;k)\in \Gamma\telque x\in K, |k|=1\} = UT^*M|_K\cap \Gamma$, for any compact
subset $K\subset M$. Since $M$ can be covered by a countable
union of compact sets, we assume without loss of generality
that $K$ is covered by a single
chart $(U,\psi)$ such that $\psi(K)\subset Q$, 
where $Q=[-1,1]^n$ is an $n$-dimensional
cube. We hence can assume w.l.g. that $M=\mathbb{R}^n$ and $UT^*M = \mathbb{R}^n\times S^{n-1}$.
It will be convenient to use the norm $d_\infty$ on $\mathbb{R}^n$, defined by: $\forall x,y\in \mathbb{R}^n$, $d_\infty(x,y):=
\sup_{1\leq i\leq n}|x_i-y_i|$. We denote
by $\overline{B}_\infty(x,r) =\{y\in \mathbb{R}^n\telque d_\infty(x,y)\leq r\}$ the closed
ball of radius $r$ for this norm.
We also denote the restriction of $d_\infty$ to $S^{n-1}\times S^{n-1}$ by the same letter and,
lastly, for $(x;\xi),(y;\eta)\in UT^*\mathbb{R}^n$ we set 
$d_\infty((x;\xi),(y;\eta)) = \sup (d_\infty(x,y),d_\infty(\xi,\eta))$.

We define cubes centered at rational points in $Q$:
let $q_{j}=[-\frac{1}{2^j},\frac{1}{2^j}]^n= \overline{B}_\infty(0,2^{-j})$ and 
$q_{j,m}=\frac{m}{2^j}+q_j = \overline{B}_\infty(2^{-j}m,2^{-j})$, where $m\in \mathbb{Z}^n \cap 2^j Q$.
In other words, the center of $q_{j,m}$ runs over a hypercubic lattice with
coordinates $(2^{-j}m_1,\dots,2^{-j}m_n)$, where
$-2^j \le m_i \le 2^j$. Note that, for each non-negative integer $j$,
the hypercubes $q_{j,m}$ overlap and cover $Q$:
\begin{eqnarray}
Q &\subset & \bigcup_{m\in \bbZ^n\cap 2^jQ} q_{jm}.
\label{QUqjm}
\end{eqnarray}

Denote by $\underline{\pi}:UT^*\mathbb{R}^n\longrightarrow \mathbb{R}^n$
and $\overline{\pi}:UT^*\mathbb{R}^n\longrightarrow S^{n-1}$
the projection maps defined by $\underline{\pi}(x;k)=x$ and $\overline{\pi}(x;k) = k$.
We define $F_{j,m}=\underline{\pi}^{-1}(q_{j,m})\simeq q_{j,m}\times S^{n-1}$
(see Fig.~\ref{figexhaustion}).
The set $\overline{\pi}(F_{j,m}\cap U\Gamma_K)$ is compact because the projection
$\overline{\pi}$ is continuous and $F_{j,m}\cap U\Gamma_K$ is compact.
For any positive integer $\ell$, define the compact set
$C_{j,m,\ell} = \big\{\eta\in S^{n-1}\telque 
d_\infty\big(\overline{\pi}(F_{j,m}\cap U\Gamma_K),\eta\big) \ge
1/\ell\big\}$.
This is the set of points of the sphere which are at least
at a distance $1/\ell$ from the projection of
the slice of $U\Gamma$ inside 
$F_{j,m}$ (see fig.~\ref{figexhaustion}).
We have
\begin{eqnarray*}
\bigcup_{\ell>0} C_{j,m,\ell} &=& S^{n-1}\backslash \overline{\pi}(F_{j,m}\cap
U\Gamma_K).
\end{eqnarray*}
Indeed, by definition, any element of $C_{j,m,\ell}$ is
in $S^{n-1}$ and not in $\overline{\pi}(F_{j,m}\cap U\Gamma_K)$,
conversely, the compactness of $\overline{\pi}(F_{j,m}\cap U\Gamma_K)$ implies
that any point $(x;\xi)$ in $S^{n-1}\backslash \overline{\pi}(F_{j,m}\cap U\Gamma_K)$ 
is at a finite distance $\delta$ from 
$\overline{\pi}(F_{j,m}\cap U\Gamma_K)$. If we take $\ell > 1/\delta$, 
we have $(x;\xi)\in C_{j,m,\ell}$.
Note that all $C_{j,m,\ell}$ are empty if
$\overline{\pi}(F_{j,m}\cap U\Gamma_K) = S^{n-1}$.
With this notation we can now state
\begin{lem}\label{approxlemma1}
\begin{eqnarray}\label{identity-qK=UmoinsGamma}
\bigcup_{j,m,\ell} q_{j,m}\times C_{j,m,\ell} &=& 
UT^*M|_K\backslash U\Gamma_K,\
\end{eqnarray}
and, by denoting 
$V_{j,m,\ell} = \{k\in \mathbb{R}^n\setminus \{0\}\telque 
k/|k|\in C_{j,m,\ell}\}$,
\begin{eqnarray}\label{identity-qV=TmoinsGamma}
\bigcup_{j,m,\ell} q_{j,m}\times V_{j,m,\ell} &=& 
(T^*M\backslash\Gamma)|_K.
\end{eqnarray}
\end{lem}
\begin{figure}
\begin{center}
\includegraphics[width=10.0cm]{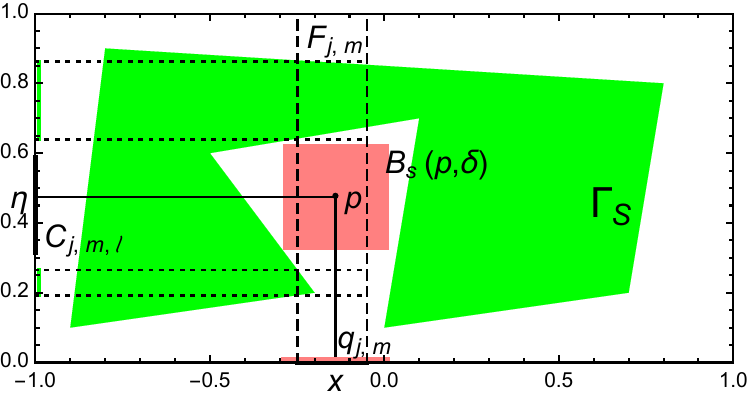}
\caption{Picture of the case $n=2$, where only one dimension
of the cube $[-1,1]^2$ is shown and the circle
$S^1$ is represented by the vertical segment $[0,1]$.
The large surface is $U\Gamma_K$, the small square around $p$ is
the ball $\overline{B}_\infty(p,\delta)$. 
We see that $q_{j,m}$ contains $x$ and is contained
in $\underline{\pi}(\overline{B}_\infty(p,\delta))$.
  \label{figexhaustion}}
\end{center}
\end{figure}
\begin{proof}
We first prove the inclusion $\subset$ in (\ref{identity-qK=UmoinsGamma}). Let 
$(x;k)\in \bigcup_{j,m,\ell} q_{j,m}\times C_{j,m,\ell}$, this means that there exist
$j\in \mathbb{N}$, $m\in \mathbb{Z}^n\cap 2^jQ$ and $\ell\in \mathbb{N}^*$
such that $(x;k)\in q_{j,m}\times C_{j,m,\ell}$. Hence $(x;k)\in F_{j,m}$ and,
by definition of $C_{j,m,\ell}$, $d_\infty(\overline{\pi}(F_{j,m}\cap U\Gamma_K),k)\geq 1/\ell$, which
implies that $(x;k)\notin U\Gamma_K$.

Let us prove the reverse inclusion $\supset$. Let $(x;k)\in UT^*M|_K\setminus U\Gamma_K$. Since this set
is open, there exists some $\delta>0$ such that $\overline{B}_\infty((x;k),\delta)\subset UT^*M|_K\setminus U\Gamma_K$.
Let $j\in \mathbb{N}^*$ s.t. $2^{-j+1}<\delta$.
Because the sets $q_{j,m}$ cover $Q$
(see eq.~\eqref{QUqjm}), there 
is an $m$ such that $x\in q_{jm}$.
Moreover, $\forall y\in q_{j,m}$, 
we have $d_\infty(x,y) \leq d_\infty(x,2^{-j}m) + d_\infty(2^{-j}m,y)\leq 2^{-j}+ 2^{-j}<\delta$,
i.e. $y\in \overline{B}_\infty(x,\delta)$. Hence
 $q_{j,m}\subset \overline{B}_\infty(x,\delta)$.
We deduce that 
\[
q_{j,m}\times \overline{B}_\infty(k,\delta)\subset \overline{B}_\infty(x,\delta)\times \overline{B}_\infty(k,\delta)
= \overline{B}_\infty((x;k),\delta)\subset UT^*M|_K\setminus U\Gamma_K.
\]
This means that $(q_{j,m}\times \overline{B}_\infty(k,\delta)) \cap U\Gamma_K = \emptyset$ or equivalently
$F_{j,m}\cap \overline{\pi}^{-1}(\overline{B}_\infty(k,\delta))\cap U\Gamma_K = \emptyset$. The latter
inclusion implies that $\overline{B}_\infty(k,\delta)\cap \overline{\pi}(F_{j,m}\cap U\Gamma_K) = \emptyset$.
In other words, $d_\infty(k,\overline{\pi}(F_{j,m}\cap U\Gamma_K)) >\delta$. Hence by choosing $\ell\in \mathbb{N}^*$
s.t. $1/\ell \leq \delta$, we deduce that $d_\infty(k,\overline{\pi}(F_{j,m}\cap U\Gamma_K) > 1/\ell$, i.e.
that $k\in C_{j,m,\ell}$. Thus we conclude that $(x;k)\in q_{j,m}\times
C_{j,m,\ell}$ and Eq.~\eqref{identity-qK=UmoinsGamma} is proved. 

To prove (\ref{identity-qV=TmoinsGamma}), we notice that,
because of the conic property of $\Gamma$, each $C_{j,m,\ell}$
corresponds to a unique $V_{j,m,\ell}$. 
\end{proof}

Any nonnegative smooth function
$\psi$ supported on $[-3/2,3/2]^n$ and such that
$\psi(x)=1$ for $x\in [-1,1]^n$ enables us to
define scaled and shifted functions
$\psi_{j-1,m}(x)=\psi(2^j(x-m))$ supported on $q_{j-1,m}$
and equal to 1 on $q_{j,2m}$.
If $C_{j,m,\ell}$ is not empty,
we denote by $\alpha_{j,m,\ell}:S^{n-1}\to\bbR$ a smooth
function supported on $C_{j,m,\ell}$ and equal to 1
on $C_{j,m,\ell+1}$.  Note that, if 
$C_{j,m,\ell}$ is a proper subset of $S^{n-1}$, then
it is strictly included in $C_{j,m,\ell+1}$.

\subsection{Equivalence of topologies}
Grigis and Sj\"ostrand stated~\cite[p.~80]{Grigis} that
if we have a family $\chi_\alpha$ of test functions and closed
cones $V_\alpha$ such that
$(\supp \chi_\alpha\times V_\alpha)\cap\Gamma=\emptyset$
and $\cup_\alpha \{(x,k)\telque \chi_\alpha(x)\not=0
\text{ and } k\in \mathring{V}_\alpha\} =\Gamma^c$, then the topology
of $\calD'_\Gamma$ is the topology given by the
seminorms of the weak topology and the seminorms
$||\cdot||_{N,V_\alpha,\chi_\alpha}$. 
By covering $M$ with a countable
family
of compact sets $K_i$ described 
in section 
\ref{exhaustionsubsec}, we see that
Lemma \ref{approxlemma1}
gives us a family
of indices $\alpha=(i,j,\ell)$, functions
$\chi_{j,m,\ell}=\psi_{j,m}$ and
cones $V_{j,m,\ell}$ adapted to $K_i$
such that the conditions
of the Grigis-Sj\"ostrand
lemma are satisfied. Therefore,
the normal topology
is described by the
seminorms 
of the strong topology
of $\mathcal{D}^\prime(\Omega)$
and by the countable family
$(i,j,m,\ell)$ of seminorms.

\subsection{Topological equivalence $C^\infty(X)$ 
and $\calD'_\emptyset$}
\label{topequiCinfty}

As an application of the previous lemma, we show
\begin{lem}
\label{topequlem}
The spaces $C^\infty(X)$ and
$\calD'_\emptyset$ are topologically isomorphic.
\end{lem}
\begin{proof}
The two spaces are identical as vector spaces because
a distribution $u$ whose wave front set is empty
is smooth everywhere, since its singular
support $\mathrm{sing\,supp}(u)=\underline{\pi}(\WF(u))$~\cite[p.~254]{HormanderI}
is empty~\cite[p.~42]{HormanderI}, and 
a distribution is a smooth function if and only if
its singular support is empty.

To prove the topological equivalence, we must show that
the two inclusions 
$\calD'_\emptyset\hookrightarrow C^\infty(X)$
and
$C^\infty(X)\hookrightarrow \calD'_\emptyset$
are continuous.
Recall that a system of semi-norms defining the
topology of $C^\infty(X)$ is
$\pi_{m,K}$, where $m$ runs over the integers and
$K$ runs over the compact
subsets of $X$~\cite[p.~88]{Schwartz-66}.
By a straightforward estimate, we 
obtain:
\begin{eqnarray*}
\pi_{m,K}(\varphi) & \le & C_{n} (2\pi)^{-n} 
\sum_{|\alpha|\le m} 
\sup_{k\in \bbR^n}
(1+ \vert k\vert^2)^p |k^\alpha \widehat{\varphi\chi}(k)|\\
&\le & C_{n} (2\pi)^{-n} 
\binom{m+n}{n}
||\varphi||_{m+2p,\bbR^n,\chi},
\end{eqnarray*}
where $\chi\in \calD(X)$ is equal to one on a compact
set whose interior contains $K$ where 
we used $(1+\vert k\vert^2)\le (1+|k|)^2$,
$|k^\alpha|\le (1+|k|)^m$ and
$\sum_{|\alpha|\le m} = \binom{m+n}{n}$.
Thus, every seminorm of $C^\infty(X)$ is bounded
by a seminorm of $\calD'_\emptyset$ and the injection
$\calD'_\emptyset\hookrightarrow C^\infty(X)$ is continuous.

Conversely, for any closed conic set $V$ and any
$\chi\in \calD(X)$ we have 
$||\varphi||_{N,V,\chi}\le ||\varphi||_{N,\bbR^n,\chi}$.
Thus, it is enough to estimate
$||\varphi||_{N,\bbR^n,\chi}$.
We also find that, for any 
integer $N$ and $\alpha=0$,
$\sup_{k\in\bbR^n}
(1+k^2)^N |\widehat{\varphi\chi}(k)|
\le |K|2^N \pi_{2N,K}(\varphi\chi)$,
where $K$ is the support of $\varphi$.
Then, the relation 
$1+|k|\le 2(1+|k|^2)$ and 
application of the Leibniz rule 
give us
$||\varphi||_{N,\bbR^n,\chi}\le 
|K|8^N \pi_{2N,K}(\chi) \pi_{2N,K}(\varphi)$ and
the seminorms $||\cdot||_{N,V,\chi}$ are
controlled by the seminorms of $C^\infty(X)$.
For the seminorms of $\calD'(X)$,
it is well known that the inclusion 
$C^\infty(X)\hookrightarrow \calD'(X)$ is
continuous~\cite[p.~420]{Schwartz-66} when
$\calD'(X)$ has its strong topology. 
Therefore, it is also continuous when
$\calD'(X)$ has its weak topology
and we proved that
$C^\infty(X)$ and $\calD'_\emptyset$ are
topologically isomorphic, where
$\calD'_\emptyset$ can be equipped with the H\"ormander or the
normal topology.
\end{proof}

\subsection*{Acknowledgements}
We are very grateful to Yoann Dabrowski for his generous help
in the functional analysis parts of the paper.
We thank Camille Laurent-Gengoux for discussions
about the geometrical aspects of the pull-back.
The research of N.V.Dang was partly supported by the Labex CEMPI (ANR-11-LABX-0007-01).

\end{document}